\documentclass{article}
\usepackage{fullpage}

\usepackage[utf8]{inputenc}
\usepackage{amsthm,amssymb}
\usepackage[leqno]{amsmath}
\usepackage{thmtools,mathtools}
\usepackage{caption,subcaption}
\usepackage{epsfig,graphicx,graphics,color,tikz,tikz-cd,tcolorbox}
\usetikzlibrary{calc,decorations.pathmorphing,shapes}
\usepackage{calc}
\usepackage{enumerate,enumitem}
\usepackage[sort]{cite}
\usepackage{xspace}
\usepackage{bm}
\usepackage{csquotes}
\usepackage{hyperref}
\hypersetup{
    colorlinks=true,       
    linkcolor=blue,        
    citecolor=magenta,         
    filecolor=magenta,     
    urlcolor=cyan,         
    linktoc=all
}

\usepackage{xcolor}
\newcommand{\randomcolor}{%
\definecolor{randomcolor}{RGB}
   {
    \pdfuniformdeviate 256,
    \pdfuniformdeviate 256,
    \pdfuniformdeviate 256
   }%
  \color{randomcolor}%
}

\newlist{typelist}{enumerate}{1} 
\setlist[typelist]{left=0pt, itemsep=1pt, 
         label=(\textbf{Type \arabic*}), ref=\arabic*}

\newlength{\bibitemsep}
\newlength{\bibparskip}
\let\oldthebibliography\thebibliography
\renewcommand\thebibliography[1]{%
  \oldthebibliography{#1}%
  \setlength{\parskip}{\bibitemsep}%
  \setlength{\itemsep}{\bibparskip}%
}

\theoremstyle{plain}
\newtheorem{thm}{Theorem}[section]
\newtheorem{lem}[thm]{Lemma}
\newtheorem{cor}[thm]{Corollary}
\newtheorem{cl}[thm]{Claim}
\newtheorem{prop}[thm]{Proposition}
\newtheorem{conj}{Conjecture}

\theoremstyle{definition}

\newtheorem{rem}[thm]{Remark}

\def\final{0}  
\def\iflong{\iffalse}
\ifnum\final=0  
\newcommand{\knote}[1]{{\color{red}[{\tiny \textbf{Kristóf:} \bf #1}]\marginpar{\color{red}*}}}
\newcommand{\tnote}[1]{{\color{blue}[{\tiny \textbf{Tamás:} \bf #1}]\marginpar{\color{blue}*}}}
\newcommand{\bnote}[1]{{\randomcolor[{\tiny \textbf{Bence:} \bf #1}]\marginpar{\randomcolor*}}}
\else 
\newcommand{\knote}[1]{}
\newcommand{\tnote}[1]{}
\newcommand{\bnote}[1]{}
\fi

\DeclareMathOperator\car{cr}

\DeclareMathOperator\conv{conv}

\newcommand{\bR}{\mathbb{R}}
\newcommand{\bK}{\mathbb{K}}
\newcommand{\bF}{\mathbb{F}}

\newcommand{\cB}{\mathcal{B}}
\newcommand{\cC}{\mathcal{C}}
\newcommand{\cF}{\mathcal{F}}

\newcommand{\cI}{\mathcal{I}}

\newcommand{\cO}{\mathcal{O}}
\newcommand{\cP}{\mathcal{P}}
\newcommand{\cT}{\mathcal{T}}
\newcommand{\cX}{\mathcal{X}}
\newcommand{\cY}{\mathcal{Y}}
\newcommand{\cZ}{\mathcal{Z}}

\newcommand{\dy}{$\Delta$\,-\,{\large \texttt{Y}}\xspace}
\newcommand{\yd}{{\large \texttt{Y}}\,-\,$\Delta$\xspace}

\title{Reconfiguration of Basis Pairs in Regular Matroids}

\author{
 Kristóf Bérczi\thanks{MTA-ELTE Momentum Matroid Optimization Research Group and HUN-REN–ELTE Egerváry Research Group, Department of Operations Research, Eötvös Loránd University, Budapest, Hungary.\\Emails: \texttt{kristof.berczi@ttk.elte.hu, matben@student.elte.hu, tamas.schwarcz@ttk.elte.hu}.}
 \and
 Bence Mátravölgyi\footnotemark[1]
 \and
 Tamás Schwarcz\footnotemark[1]
}

\date{}

\begin{document}
\maketitle
\thispagestyle{empty}

\begin{abstract} 
In recent years, combinatorial reconfiguration problems have attracted great attention due to their connection to various topics such as optimization, counting, enumeration, or sampling. One of the most intriguing open questions concerns the exchange distance of two matroid basis sequences, a problem that appears in several areas of computer science and mathematics. In 1980, White~\cite{white1980unique} proposed a conjecture for the characterization of two basis sequences being reachable from each other by symmetric exchanges, which received a significant interest also in algebra due to its connection to toric ideals and Gr\"obner bases. In this work, we verify White's conjecture for basis sequences of length two in regular matroids, a problem that was formulated as a separate question by Farber, Richter, and Shank~\cite{farber1985edge} and Andres, Hochst\"{a}ttler, and Merkel~\cite{andres2014base}. Most of previous work on White's conjecture has not considered the question from an algorithmic perspective. We study the problem from an optimization point of view: our proof implies a polynomial algorithm for determining a sequence of symmetric exchanges that transforms a basis pair into another, thus providing the first polynomial upper bound on the exchange distance of basis pairs in regular matroids. As a byproduct, we verify a conjecture of Gabow from 1976~\cite{gabow1976decomposing} on the serial symmetric exchange property of matroids for the regular case.

\medskip


\end{abstract}

\newpage
\pagenumbering{roman}
\tableofcontents
\newpage
\pagenumbering{arabic}
\setcounter{page}{1}
\section{Introduction}
\label{sec:intro}

The \emph{basis exchange axiom} of matroids implies that for any pair $X,Y$ of bases, there exists a sequence of exchanges that transforms $X$ into $Y$. White~\cite{white1980unique} studied the analogous problem for basis sequences instead of single bases. Let $\cX=(X_1,\dots,X_k)$ be a sequence of -- not necessarily disjoint -- bases of a matroid, and let $e\in X_i-X_j$ and $f\in X_j-X_i$ with $1\leq i<j\leq k$ be such that both $X_i-e+f$ and $X_j+e-f$ are bases. Then, the sequence $\cX'=(X_1,\dots,X_{i-1},X_i-e+f,X_{i+1},\dots,X_{j-1},X_j+e-f,X_{j+1},\dots,X_k)$ is obtained from $\cX$ by a \emph{symmetric exchange}. Two sequences $\cX$ and $\cY$ are called \emph{equivalent} if $\cY$ can be obtained from $\cX$ by a composition of symmetric exchanges. The question naturally arises: what is the characterization of two basis sequences being equivalent? 

There is an easy necessary condition for the equivalence of two sequences $\cX$ and $\cY$: since a symmetric exchange does not change the number of bases in the sequence that contain a given element, the union of the members of $\cX$ must coincide with the union of the members of $\cY$ as multisets. Motivated by this observation, $\cX$ and $\cY$ are called \emph{compatible} if $|\{i\mid e\in X_i, i\in\{1,\dots,k\}\}|=|\{i\mid e\in Y_i,i\in\{1,\dots,k\}\}|$ for every $e \in E$, where $E$ denotes the ground set of the matroid. White~\cite{white1980unique} conjectured that compatibility is not only necessary but also sufficient for two sequences to be equivalent.

\begin{conj}[White]\label{conj:white}
Two basis sequences $\cX$ and $\cY$ of the same length are equivalent if and only if they are compatible.
\end{conj}

Conjecture~\ref{conj:white} received a significant interest also in algebra due to its connection to toric ideals and Gröbner bases, see~\cite{blasiak2008toric} and~\cite[Chapter 13]{michalek2021invitation} for further details. However, despite all the efforts, White's conjecture remains open even for sequences of length two. In this special setting, Farber, Richter, and Shank~\cite{farber1985edge} verified the statement for graphic and cographic matroids, and noted that their proof does not seem to generalize for regular matroids. Andres, Hochst\"{a}ttler and Merkel~\cite{andres2014base} formulated White's conjecture for regular matroids as a separate question, and noted that Seymour's decomposition theorem~\cite{seymour1980decomposition} might help to find a proof for it.

White's conjecture has no implications on the minimum number of exchanges needed to transform two equivalent sequences into each other, called their \emph{exchange distance}. For basis pairs, Gabow~\cite{gabow1976decomposing} formulated the following problem, later stated as a conjecture by Wiedemann \cite{wiedemann1984cyclic} and by Cordovil and Moreira~\cite{cordovil1993bases}, and posed as an open problem in Oxley's book~\cite[Conjecture~15.9.11]{oxley2011matroid}.

\begin{conj}[Gabow] \label{conj:gabow}
Let $X_1$ and $X_2$ be disjoint bases of a rank-$r$ matroid $M$. Then, the exchange distance of $(X_1,X_2)$ and $(X_2,X_1)$ is $r$.
\end{conj}

Note that Conjecture~\ref{conj:gabow} would imply Conjecture~\ref{conj:white} for sequences of the form $(X_1,X_2)$ and $(X_2,X_1)$. Since the rank of the matroid is a trivial lower bound on the minimum number of exchanges needed to transform $(X_1,X_2)$ into $(X_2,X_1)$, the essence of Gabow's conjecture is that rank many steps might always suffice. This also implies that the conjecture can be rephrased as a generalization of the symmetric exchange axiom as follows: If $X_1$ and $X_2$ are bases of the same matroid, then there are orderings $X_1=(x^1_1,\dots,x^1_r)$ and $X_2=(x^2_1,\dots,x^2_r)$ such that $\{x^1_1,\dots,x^1_i,x^2_{i+1},\dots,x^2_r\}$ and $\{x^2_1,\dots,x^2_i,x^1_{i+1},\dots,x^1_r\}$ are bases for $i=0,\dots,r$. This property is often referred to as \emph{serial symmetric exchange property}. 

\medskip

The focus of this paper is on regular matroids, a fundamental class that generalizes graphic and cographic matroids. Our main tool is Seymour's celebrated decomposition theorem, which gives a method for decomposing any regular matroid into matroids which are either graphic, cographic, or isomorphic to a simple 10-element matroid. Regular matroids play a crucial role in both matroid theory and optimization, since those are exactly the matroids that can be represented over $\bR$ by totally unimodular matrices~\cite[Theorem 6.6.3]{oxley2011matroid}. This connection has far reaching implications, e.g.\ the fastest known algorithm for testing total unimodularity of a matrix is based on the ability to find such a decomposition if one exists~\cite{truemper1990decomposition}. 

Interest in exchange properties of matroids originally arose in part from the fact that they serve as an abstraction of pivot algorithms of linear algebra~\cite{white1980unique,greene1975some}. In the past decades, however, problems appeared in many different areas of computer science and mathematics that are actually based on exchange properties of matroid bases, though these problems have never been explicitly linked together in previous work. Implicitly, one of the goals of the paper is to draw attention to these connections.

\subsection{The Role of Equivalent Sequences}
\label{sec:motivation}

Though finding a sequence of symmetric exchanges between basis sequences may seem to be a structural question purely on matroids, it has been identiﬁed as the key ingredient in a range of problems. In what follows, we give an overview of main applications where the reconfiguration of basis sequences shows up. 

\subsubsection*{Sampling Common Bases}

Mihail and Vazirani conjectured that the basis exchange graph of any matroid has edge expansion at least one, see~\cite{feder1992balanced}. The motivation behind the conjecture was to solve the problem of approximately sampling from bases of a matroid. After the appearance of the conjecture, a long line of work concentrated on designing approximation algorithms to count the number of bases of a matroid, and efficient sampling algorithms were developed for various special classes. Most of these results relied on the Markov Chain Monte Carlo technique: for any matroid, the basis exchange property defines a natural random walk which mixes to the uniform distribution over all bases, also known as ``down-up'' random walk in the context of high-dimensional expanders. In a breakthrough result, Anari et al.~\cite{anari2019log} verified the conjecture of Mihail and Vazirani, and thus gave an efficient approximate sampling algorithm for all matroids.

Sampling common bases of two matroids is also of interest. In~\cite{anari2021log}, Anari, Gharan and Vinzant gave a deterministic polynomial time $2^{O(r)}$-approximation algorithm for the number of common bases of any two matroids of rank $r$. Unlike in the case of a single matroid, the intersection of two matroids does not satisfy the exchange property. Even worse, there are examples showing that the symmetric difference of a common basis with any other common basis might be large, hence there is no hope for defining a simple down-up-type random walk in general. 

Partitions of the ground set of a matroid $M$ into two disjoint bases can be identified with common bases of $M$ and its dual $M^*$. From this perspective, White's conjecture states that there is a sequence of exchanges between any pair of common basis of $M$ and $M^*$. Thus verifying Conjecture~\ref{conj:white} would open up the possibility for a natural down-up random walk for matroid intersection in the special case when the two matroids are dual to each other.

\subsubsection*{Equitability of Matroids}

There are several further problems that aim at a better understanding of the structure of bases. The probably most appealing one is the \emph{Equitability Conjecture} for matroids that provides a relaxation of both Conjecture~\ref{conj:white} and Conjecture~\ref{conj:gabow}. A matroid whose ground set $E$ partitions into disjoint bases is called \emph{equitable} if for any set $Z\subseteq E$, there exists a partition into disjoint bases $E=X_1\cup\dots\cup X_k$ such that $\lfloor|Z|/k\rfloor\leq|X_i\cap Z|\leq\lceil|Z|/k\rceil$. The Equitability Conjecture states that every matroid whose ground set partitions into disjoint bases is equitable. 

The existence of such a partition would follow from both Conjecture~\ref{conj:white} and Conjecture~\ref{conj:gabow}. To see this, first observe that it suffices to consider the case $k=2$, since for general $k$ the statement then follows by repeated application of the problem restricted to $X_i\cup X_j$ for $1\leq i<j\leq k$; see \cite[Discussion page]{egres_open_equit} for details. Then, for any partition of the ground set into two bases $X_1$ and $X_2$, both Conjecture~\ref{conj:white} and Conjecture~\ref{conj:gabow} imply the existence of a sequence of symmetric exchanges that transforms $(X_1,X_2)$ into $(X_2,X_1)$. One of the basis pairs of the sequence thus obtained must satisfy $\lfloor |Z|/2\rfloor \leq |X_i\cap Z|\leq \lceil |Z|/2\rceil$ for $i=1,2$.

Apart from the matroid classes for which Conjecture~\ref{conj:white} or \ref{conj:gabow} was settled, the Equitability Conjecture was verified for base orderable matroids~\cite{fekete2011equitable} only. It is worth mentioning that equitable partitions are closely related to fair representations, introduced by Aharoni, Berger, Kotlar, and Ziv~\cite{aharoni2017faira}. For further details, we refer the interested reader to~\cite{berczi2022exchange,fekete2011equitable,egres_open_equit}.

\subsubsection*{Fair Allocations}

Fair allocation of indivisible goods has received a significant interest in the past decade. In such problems, the goal is to find an allocation of a set $E$ of $m$ indivisible items among $n$ agents so that each agent finds the allocation fair. The degree of equality can be measured in various ways, and envy-freeness, introduced by Foley~\cite{foley1966resource} and Varian~\cite{varian1973equity}, is one of the most natural fairness concepts. An allocation is considered to be \emph{envy-free} (EF) if each agent finds the value of her bundle at least as much as that of any other agent. 

Though envy-freeness imposes a very natural criterion for the fairness of an allocation, such a solution may not exist. As a workaround, several relaxations have been proposed, including \emph{envy-freeness up to one good} (EF1). Biswas and Barman~\cite{biswas2018fair} and Dror, Feldman, and Segal-Halevi~\cite{dror2023fair} studied the existence of EF1 allocations under matroid constraints. If the agents have different matroid constraints, then such an allocation might not exist even for two agents with identical valuations. Therefore, it is natural to consider identical matroid constraints. Given a matroid $M$ on the set of items, an allocation is called \emph{feasible} if the bundle of each agent forms an independent set of $M$. Deciding the existence of a feasible EF1 allocation as well as finding one algorithmically are interesting open problems that have only been solved for very restricted cases: for partition matroids~\cite{biswas2018fair}, for base orderable matroids with identical valuations~\cite{biswas2018fair}, for base orderable matroids with two agents~\cite{dror2023fair}, and for base orderable matroids with three agents and binary valuations, i.e.\ when each item has value zero or one for every agent~\cite{dror2023fair}.

The problem remains open even when the agents share the same binary valuation. In such a case, there exists a $Z\subseteq E$ such that the value of any subset $X$ of items is $|X\cap Z|$ for every agent, and a feasible EF1 allocation corresponds to a partition of the ground set into $n$ independent sets $X_1,\dots,X_n$ such that $\lfloor |Z|/n\rfloor \leq |X_i\cap Z|\leq \lceil |Z|/n\rceil$ for every $1\leq i\leq n$. The existence of such an allocation would follow from the Equitability Conjecture, and hence from both Conjecture~\ref{conj:white} and Conjecture~\ref{conj:gabow}. The only difference compared to the setting of the Equitability Conjecture is that here we seek for a partition into independent sets instead of bases. However, by possibly taking the direct sum of $M$ with a free matroid and then truncating it, the sets $X_i$ can be assumed to form bases of the matroid. 

\subsubsection*{Carathéodory Rank of Matroid Base Polytopes}

A polyhedron $P\subseteq\bR^n$ has the \emph{integer decomposition property} if for every positive integer $k$, every integer vector $x$ in $kP$ can be written as $x=\sum_{i=1}^t\lambda_i x_i$, where each $\lambda_i$ is a positive integer together satisfying $\sum_{i=1}^t\lambda_i=k$, and each $x_i$ is an integer vector in $P$. This notion was introduced by Baum and Trotter~\cite{baum1978integer}, and plays an important role in the theory and application of integer programming \cite[Section 22.10]{schrijver1998theory} as well as in the study of toric varieties~\cite{viet1997normal}, \cite[Chapter 2]{cox2011toric}). For a polyhedron $P$ with the integer decomposition property, the smallest number $\car(P)$ such that we can take $t\leq \car(P)$ for every $k$ and every $x\in kP$ is called the \emph{Carath\'eodory rank} of $P$. In~\cite{cunningham1984testing}, Cunningham asked whether a sum of bases in $M$ can always be written as a sum using at most $n$ bases, where $n$ is the cardinality of the ground set, which was answered in the affirmative by Gijswijt and Regts~\cite{gijswijt2012polyhedra}. That is, $\car(P)\leq n$ holds where $P$ is the convex hull of the incidence vectors of bases of $M$. 

For a polytope $P\subseteq\bR^n$ with vertices $v_1,\dots,v_p$, a \emph{triangulation} $\cT$ of $P$ is a collection of simplices on the vertices of $P$ such that (i) if $T\in\cT$ then all faces of $T$ are in $\cT$, (ii) if $T_1,T_2\in\cT$ then $T_1\cap T_2$ is a face of both $T_1$ and $T_2$, and (iii) $\bigcup_{T\in \cT}\conv(T)=P$. A triangulation $\cT$ is \emph{unimodular} if the volume of every highest dimensional simplex of $\cT$ are the same. As a geometric variant of Conjecture~\ref{conj:white}, Haws~\cite{haws2009matroid} conjectured that every matroid base polytope has a unimodular triangulation. One motivation behind the conjecture was that the existence of such a triangulation implies a bound of $n$ on the Carath\'eodory rank of a connected matroid base polytope, a result that was proved only later in \cite{gijswijt2012polyhedra}. Recently, Backman and Liu~\cite{backman2023regular} verified Haws' conjecture by showing that every matroid base polytope admits a regular unimodular triangulation.

\subsubsection*{Toric Ideals}

Describing minimal generating set of a toric ideal is a well-studied and difficult problem. Consider a matroid $M=(E,\cB)$ where $E$ denotes the ground set and $\cB$ is the family of bases of $M$. For a field $\bK$, let $S_M$ denote the polynomial ring $\bK[y_B\mid B\in\cB]$. The \emph{toric ideal associated to $M$} is the kernel of the $\bK$-homomorphism $\varphi_M\colon S_M\to\bK[x_e:e\in E]$ given by $y_B\mapsto\prod_{e\in B} x_e$. Assume now that the basis pair $(Y_1,Y_2)$ is obtained from $(X_1,X_2)$ by a symmetric exchange, that is, $Y_1=X_1-e+f$ and $Y_2=X_2+e-f$ for some $e\in X_1- X_2$ and $f\in X_2- X_1$. Then, the quadratic binomial corresponding to the symmetric exchange is $y_{X_1}y_{X_2}-y_{Y_1}y_{Y_2}$. It is not difficult to see that such binomials belong to the ideal $I_M$, and White~\cite{white1980unique} conjectured that they in fact generate $I_M$. 

More precisely, White stated three conjectures of growing diﬃculty. Using the notation of \cite{white1980unique}, two sequences $\cX,\cY$ of bases of equal length are in relation $\sim_1$ if $\cY$ can be obtained from $\cX$ by a composition of symmetric exchanges, in relation $\sim_2$ if $\cY$ can be obtained from $\cX$ by a composition of symmetric exchanges and permutations of the order of the bases, and in relation $\sim_3$ if $\cY$ can be obtained from $\cX$ by a composition of symmetric exchanges of subsets. Let $TE(i)$ denote the class of matroids for which every two compatible sequences $\cX,\cY$ are in relation $\cX\sim_i\cY$. In algebraic terms, a matroid belongs to $TE(3)$ if and only if its is generated by quadratic binomials, and it belongs to $TE(2)$ if and only if its toric ideal is generated by quadratic binomials corresponding to symmetric exchanges. Property $TE(1)$ is a counterpart of $TE(2)$ for the noncommutative polynomial ring $S_M$. The three conjectures of White state that $TE(i)$ is the class of all matroids for $i=1,2,3$. In particular, Conjecture~\ref{conj:white} corresponds to the choice $i=1$.

\subsubsection*{Reconfiguration Problems}

In combinatorial reconfiguration problems, the goal is to study the solution space of an underlying combinatorial optimization problem. The solution space can be represented by a graph, where vertices correspond to feasible solutions and there is an edge between two vertices if the corresponding solutions can be obtained from each other by an elementary step, defined specifically for the given problem. Reconfiguration problems concern the \emph{reachability} of a solution from another in this graph, and if such a path exists, then finding a \emph{shortest} one between them. In recent years, such problems have attracted great attention due to their connection to various topics such as optimization, counting, enumeration, and sampling. For further details, we refer the interested reader to~\cite{van2013complexity,nishimura2018introduction}. 

The vertices of the \emph{exchange graph} of a matroid correspond to basis sequences of a given length, two vertices being connected by an edge if the corresponding basis sequences can be obtained from each other by a single symmetric exchange. In this context, Conjecture~\ref{conj:white} aims at characterizing reachability in the exchange graph, and states that the connected components are exactly the equivalence classes of compatibility. 

An analogous problem can be formulated for the intersection of two matroids, i.e.\ given two common bases of two matroids, decide if one can be obtained from the other by always changing a single element while maintaining independence in both matroids. Such a sequence of exchanges is known to exist in special cases, e.g.\ for arborescences, or more generally, for $k$-arborescences~\cite{kobayashi2023reconfiguration}. Recently, the problem was shown to be oracle hard by Kobayashi, Mahara, and Schwarcz~\cite{kobayashi2023reconfiguration}. For sequences of length two, Conjecture~\ref{conj:white} is the special case of the common basis reconfiguration problem when the two matroids are dual to each other.

\subsection{Our Results}
\label{sec:results}

Motivated by the signiﬁcance of equivalent basis sequences in various applications and by the fact that it was formulated as an interesting open problem in~\cite{farber1985edge,andres2014base}, we study the exchange distance of basis pairs in regular matroids. First, we give a polynomial upper bound on the exchange distance of compatible basis pairs, which proves Conjecture~\ref{conj:white} for sequences of length two in regular matroids. Our proof is algorithmic, which allows us to determine a sequence of symmetric exchanges that transforms a given pair of bases into another in polynomial time. As usual in matroid algorithms, we assume that the matroid is given by an independence oracle and the running time is measured by the number of oracle calls and other conventional elementary steps. For the sake of simplicity, by ``polynomial number'' of oracle calls we mean ``polynomial in the number of elements of the ground set''.

\begin{thm}\label{thm:regwhite}
Let $\cX=(X_1,X_2)$ and $\cY=(Y_1,Y_2)$ be compatible basis pairs of a regular matroid $M$ of rank $r\geq 2$. Then, there exists a sequence of symmetric exchanges that transforms $\cX$ into $\cY$, has length at most $2\cdot r^2$, and uses each element at most $4\cdot (r-1)$ times. Furthermore, such a sequence can be determined using a polynomial number of oracle calls.
\end{thm}

A fine grained analysis of the algorithm shows that the number of steps can be bounded better when the basis pairs are inverses of each other. Our second result is an improved upper bound on the exchange distance of such pairs, which proves Conjecture~\ref{conj:gabow} for regular matroids.

\begin{thm}\label{thm:reggabow}
Let $X_1,X_2$ be disjoint bases of a regular matroid $M$ of rank $r$. Then, there exists a sequence of symmetric exchanges that transforms $(X_1,X_2)$ into $(X_2,X_1)$ and has length $r$. Furthermore, such a sequence can be determined using a polynomial number of oracle calls.
\end{thm}

Our results give the ﬁrst polynomial bound on the exchange distance of basis pairs and are the first to settle the conjectures of White and Gabow in regular matroids. We hope that our paper will help proving White's conjecture in regular matroids for sequences of arbitrary length, as well as obtaining better bounds for the exchange distance of basis pairs.

\subsection{Overview of Techniques}
\label{sec:techniques}

We give a high-level overview of the proofs of Theorem~\ref{thm:regwhite} and Theorem~\ref{thm:reggabow}.

\paragraph{Connectivity and Cogirth}

The ﬁrst step in proving our main results is to deduce structural properties of regular matroids that allow us to reduce the size of the problem. First, we show that the basis pairs can be assumed to consist of disjoint bases, for if not, then the problem size can be decreased by contracting $X_1\cap X_2=Y_1\cap Y_2$. We then consider tight sets, where in a matroid $M$ over a ground set $E$ a subset $Z\subseteq E$ is \emph{tight} if $|Z|=2\cdot r_M(Z)$. We show that if the basis pairs cover a tight set, then the problem can be divided into smaller subproblems on the restriction $M|Z$ and contraction $M/Z$. As a corollary, we get that it is enough to consider 2-connected matroids. Solving the problem for the 2-sum of regular matroids is based on a similar idea, but merging the solutions for the subproblems is significantly more involved. We explain how to schedule the exchanges on the two sides of the 2-sum in such a way that the basis pairs fit together at each step, meaning that they form a basis pair of the original matroid. This observation eventually reduces the problem to the case of 3-connected matroids. While the above simplifications are well-understood, our main contribution is to show that small cocircuits can also be excluded. More precisely, we prove that if the basis pairs cover a cocircuit of size at most three, then the size of the problem can be decreased by contracting and deleting certain elements of the cocircuit. 

For almost all of these operations, we prove a stronger statement that allows certain elements not to be involved in the exchange sequence. This observation will play a crucial role in the proof by providing control over the choice of symmetric exchanges to be used. Besides reducing the problem to 3-connected matroids of cogirth at least four, all the reduction steps can be performed using a polynomial number of oracle calls which is essential for achieving an efficient algorithm.

\paragraph{Graphic Matroids}

Though it is not stated explicitly, the algorithms of \cite{farber1985edge} and \cite{blasiak2008toric} that prove White's conjecture for graphic matroids give a sequence of exchanges of length at most $(n-1)^2$. Using the fact that the union of two forests always contains a vertex of degree two or at least four vertices of degree three, we give a formal proof of this bound even under certain restrictions on the set of exchanges that can be used. To the best of our knowledge, this is the first analysis of the algorithm that proves a polynomial running time and hence might be of independent combinatorial interest. 

More precisely, let both $\cX=(X_1,X_2)$ and $\cY=(Y_1,Y_2)$ be partitions of a graph $G=(V,E)$ into two maximal forests. Furthermore, let $F\subseteq (X_1\cap Y_1)\cup(X_2\cap Y_2)$ with $|V(F)|\leq 3$; note that these edges do not have to change positions between the two bases. Building on the reductions along tight sets and cocircuits of size three, we show that there exists a sequence of exchanges of length at most $(|V|-1)^2$ that transforms $\cX$ into $\cY$ and uses none of the edges in $F$. This result will be used in the proof of Theorem~\ref{thm:regwhite} as follows: when the matroid is the 3-sum of a regular and a graphic matroid along a cycle $F$ of length three, then one can solve the two subproblems corresponding to the two sides of the 3-sum while restricting the usage of the elements of $F$, which in turn allows for an efficient merging of the sequences. 

\paragraph{Refined Decomposition Theorem}

Seymour's decomposition theorem provides a way of writing any regular matroid as the 1-, 2- and 3-sums of so-called basic matroids that are graphic, cographic, or isomorphic to $R_{10}$. Andres, Hochst\"{a}ttler and Merkel~\cite{andres2014base} already noted that such a decomposition might be helpful in proving White's conjecture for regular matroids. Unfortunately, without any further information on the structure of the decomposition, solving the problem for the basic matroids does not suffice, since it is not clear how to merge these solutions together.

To overcome these difficulties, we use a recent result by Aprile and Fiorini~\cite{aprile2022regular} that gives a refinement of Seymour's theorem. Roughly speaking, they showed that any 3-connected regular matroid distinct from $R_{10}$ admits a decomposition in which the 3-sums are not using nontrivial cuts of the graphs that correspond to cographic basic matroids. McGuiness~\cite{mcguinness2014base} gave a characterization of the dual of a 3-sum using the so-called \emph{\dy exchange} operation. We observe that applying a \dy exchange corresponding to a trivial cut of the underlying graph transforms a cographic matroid into another. With the help of these results, we can identify a graphic basic matroid that is ``sitting at the end of the decomposition''. Moreover, by relying on the reduction steps for tight sets and triads mentioned earlier, we show that the graph in question can be assumed to be 4-regular. It is truly amazing that all these results come together so nicely, thus narrowing the problem down to a case that we can then tackle.

\paragraph{An Inductive Approach}

Based on our previous observations, we write up the matroid as the 3-sum of a regular matroid and the graphic matroid of a 4-regular graph. The basis pairs of the original instance can be naturally restricted to the two sides of the 3-sum, thus resulting in smaller instances for which one can find desired exchange sequences separately. If not chosen carefully, merging these two sequences to get a solution for the original instance would require too many steps or may even be impossible. On the graphic part, however, we use the strengthening of the statement for graphic matroids which exclude certain elements to take part in the exchanges. This allows us to merge the solutions for the two sides of the 3-sum efficiently.

\subsection{Related Work}
\label{sec:related}

When restricted to sequences of length two, White's conjecture was verified for graphic and cographic matroids by Farber, Richter, and Shank~\cite{farber1985edge}, for transversal matroids by Farber~\cite{farber1989basis}, and for split matroids by Bérczi and Schwarcz~\cite{berczi2022exchange}. For sequences of arbitrary length, Blasiak~\cite{blasiak2008toric} confirmed the conjecture for graphic matroids. It is not difficult to check that the conjecture holds for a matroid $M$ if and only if it holds for its dual $M^*$, therefore Blasiak's result settles the cographic case as well. Further results include lattice path matroids by Schweig~\cite{schweig2011toric}, sparse paving matroids by Bonin~\cite{bonin2013basis}, strongly base orderable matroids by Laso{\'n} and Micha{\l}ek~\cite{lason2014toric}, and frame matroids satisfying a linearity condition by McGuinness~\cite{mcguinness2020frame}. 

Gabow~\cite{gabow1976decomposing} observed that Conjecture~\ref{conj:gabow} holds for partition matroids, transversal matroids, and matching matroids. An easy proof shows that it also holds for strongly base orderable matroids. The graphic case was independently proved by Wiedemann~\cite{wiedemann1984cyclic}, Kajitani, Ueno, and Miyano~\cite{kajitani1988ordering}, and Cordovil and Moreira~\cite{cordovil1993bases}. The cases of sparse paving and split matroids were settled in~\cite{bonin2013basis} and~\cite{berczi2022exchange}, respectively.

For the case of basis pairs, a common generalization of the conjectures of White and Gabow was proposed by Hamidoune~\cite{cordovil1993bases} stating that the exchange distance of compatible basis pairs is at most the rank of the matroid. A strengthening was proposed by Bérczi, Mátravölgyi, and Schwarcz~\cite{berczi2022weighted} who considered a weighted variant of Hamidoune's conjecture and verified it for strongly base orderable matroids, split matroids, spikes, and graphic matroids of wheel graphs.  

An ordered pair $(X_1,X_2)$ of bases of a matroid satisfies the \emph{unique exchange property} if there exists an element $e\in X_1$ for which there is a unique element $f\in X_2$ that can be symmetrically exchanged with $e$. The \emph{unique exchange graph} can be defined for sequences of bases in a straightforward manner, and White~\cite{white1980unique} conjectured that for regular matroids, the connected components of this graph are exactly the equivalence classes of compatibility. The motivation behind this conjecture comes from the study of the bracket ring of a matroid, see~\cite{white1975bracketa,white1975bracketb} for details. McGuinness~\cite{mcguinness2014base} verified that any pair of bases in a regular matroid has the unique exchange property, implying the the unique exchange graph has no isolated vertices. However, Andres, Hochst\"attler and Merkel~\cite{andres2014base} disproved the conjecture, and proposed a relaxation instead in which the element with a unique symmetrically exchangeable pair can be chosen from both $X_1$ and $X_2$. 

\paragraph{Paper Organization}

The rest of the paper is organized as follows. In Section~\ref{sec:prelim}, we recall basic definitions, notation, and some results on the decomposition of regular matroids that we will use in our proofs. In Section~\ref{sec:reductions}, we show how Conjecture~\ref{conj:white} for sequences of length two and Conjecture~\ref{conj:gabow} can be reduced to 3-connected matroids not containing cocircuits of size at most three. Then, in Section~\ref{sec:graphic}, we explain how a quadratic bound on the number of exchanges can be derived for graphs using the aforementioned reductions, and prove strengthenings of White's and Gabow's conjectures for graphic matroids. The rest of the paper is devoted to proving Theorem~\ref{thm:regwhite} and Theorem~\ref{thm:reggabow}. Our proofs rely on the regular matroid decomposition theorem of Seymour. Nevertheless, solving the problems for each matroid in the decomposition in parallel and then simply merging the solutions does not work. The key ingredient that leads to Theorem~\ref{thm:regwhite} and Theorem~\ref{thm:reggabow} is a careful combination of the solutions of these subproblems that results in a sequence of exchanges whose length is polynomially bounded. For ease of reading, we encourage first-time readers to skip the technical parts of Section~\ref{sec:reductions}.

\section{Preliminaries}
\label{sec:prelim}

\paragraph{Basic Notation and Definitions}

Given a ground set $E$, the \emph{difference} of $X,Y\subseteq E$ is denoted by $X-Y$. If $Y$ consists of a single element $y$, then $X-\{y\}$ and $X\cup \{y\}$ are abbreviated as $X-y$ and $X+y$, respectively. The \emph{symmetric difference} of $X$ and $Y$ is defined as $X\triangle Y\coloneqq (X-Y)\cup(Y-X)$. 

\paragraph{Graphs}

Throughout the paper, we consider loopless graphs that might contain parallel edges. For a graph $G=(V,E)$, the \emph{set of edges incident to a vertex $v\in V$} is denoted by $\delta_G(v)$ and the \emph{degree of $v$} is $d_G(v)=|\delta_G(v)|$. We dismiss the subscript if the graph is clear from the context. For a subset $F\subseteq E$, we denote the \emph{set of vertices of the edges in $F$} by $V(F)$. For $X\subseteq V$, we denote by $F[X]$ the \emph{set of edges in $F$ induced by $X$}. The \emph{graph obtained by deleting $F$ and $X$} is denoted by $G-F-X$.
A \emph{cut} of $G$ is a subset $F\subseteq E$ of edges whose deletion increases the number of components. A cut is \emph{trivial} if $F=\delta(w)$ for some $w\in V$ and \emph{nontrivial} otherwise. A graph is called \emph{bispanning} if its edge set can be decomposed into two spanning trees. By a classical result of Tutte~\cite{tutte1961problem} and Nash-Williams~\cite{nash1964decomposition}, a graph $G=(V,E)$ is bispanning if and only if $|E|=2\cdot|V|-2$ and $|E[X]|\leq 2\cdot |X|-2$ for every $\emptyset\neq X\subseteq V$.

\paragraph{Matroids}

For basic definitions on matroids, we refer the reader to~\cite{oxley2011matroid}. A \emph{matroid} $M=(E,\cI)$ is defined by its \emph{ground set} $E$ and its \emph{family of independent sets} $\cI\subseteq 2^E$ that satisfies the \emph{independence axioms}: (I1) $\emptyset\in\cI$, (I2) $X\subseteq Y,\ Y\in\cI\Rightarrow X\in\cI$, and (I3) $X,Y\in\cI,\ |X|<|Y|\Rightarrow\exists e\in Y-X\ s.t.\ X+e\in\cI$. Members of $\cI$ are called \emph{independent}, while sets not in $\cI$ are called \emph{dependent}. The \emph{rank} $r_M(X)$ of a set $X$ is the maximum size of an independent set in $X$. The maximal independent subsets of $E$ are called \emph{bases} and their family is usually denoted by $\cB$. If the matroid is given by its family of bases instead of independent sets, then we write $M=(E,\cB)$. The \emph{dual} of $M$ is the matroid $M^*=(E,\cI^*)$ where $\cI^*=\{X\subseteq E\mid E-X\ \text{contains a basis of $M$}\}$. For technical reasons, we allow the ground set of the matroid to be the empty set, in which case the matroid is simply the \emph{empty matroid} $M=(\emptyset, \{\emptyset\})$.

Let $\cX=(X_1,\dots,X_k)$ and $\cY=(Y_1,\dots,Y_k)$ be sequences of bases of $M$. A sequence of symmetric exchanges that transforms $\cX$ into $\cY$ is called an \emph{$\cX$-$\cY$ exchange sequence}. The \emph{width} of an exchange sequence is the maximum number of occurrences of any element in it. If the symmetric exchanges do not involve the elements in $F\subseteq E$ then the exchange sequence is called \emph{$F$-avoiding}. The \emph{exchange distance} of $\cX$ and $\cY$ is the minimum length of an $\cX$-$\cY$ exchange sequence if one exists and $+\infty$ otherwise.

A \emph{circuit} is an inclusionwise minimal dependent set, while a \emph{loop} is a circuit consisting of a single element. 
A \emph{cocircuit} is an inclusionwise minimal set that intersects every basis, or equivalently, a circuit of the dual matroid. A set is said to be \emph{coindependent} if it contains no cocircuit of the matroid, or equivalently, it is independent of the dual matroid. Two elements $e,f\in E$ are \emph{parallel} if they form a circuit of size two. A circuit of size three is called a \emph{triangle}, while a corcircuit of size three is called a \emph{triad}. A \emph{cycle} of a matroid is a (possibly empty) subset of its ground set which can be partitioned into circuits. For a matroid $M$, we denote its \emph{families of independent sets}, \emph{bases} and \emph{circuits} by $\cI(M)$, $\cB(M)$ and $\cC(M)$, respectively. Unlike in graphs, the intersection of a circuit and a cocircuit of a matroid might have odd size. Nevertheless, the intersection never consists of a single element, see e.g.~\cite[Proposition~2.1.11]{oxley2011matroid}.

\begin{lem}\label{lem:int}
Let $C$ and $T$ be a circuit and a cocircuit of a matroid $M$. Then $|C\cap T|\neq 1$.    
\end{lem}

Let $M=(E,\cI)$ be a matroid and $E',E''\subseteq E$. The \emph{restriction to $E'$} and the \emph{deletion of $E-E'$} result in the same matroid $M|E'=M\backslash(E-E')=(E',\cI')$ with independence family $\cI'=\{I\in\cI\mid I\subseteq E'\}$. The \emph{contraction to $E''$} and the \emph{contraction of $(E-E'')$} result in the same matroid $M.E''=M/(E-E'')=(E'',\cI'')$ where $\cI''=\{I\in\cI\mid I\subseteq E-E'',I\cup Z\in\cI\ \text{for any $Z\in\cI,Z\subseteq E''$}\}$. A matroid $N$ that can be obtained from $M$ by a sequence of restrictions and contractions is called a \emph{minor} of $M$. The \emph{union} or \emph{sum} of two matroids $M_1=(E,\cI_1)$ and $M_2=(E,\cI_2)$ over the same ground set is the matroid $M_\Sigma=(E,\cI_\Sigma)$ where $\cI_\Sigma=\{I\subseteq E\mid I=I_1\cup I_2\ \text{for some $I_1\in\cI_1$, $I_2\in\cI_2$}\}$. We use $M_1+M_2$ for denoting the sum of $M_1$ and $M_2$. Edmonds and Fulkerson~\cite{edmonds1965transversals} showed that the rank function of the sum of two matroids is $r_\Sigma(Z)=\min\{\sum_{i=1}^2 r_i(X)+|Z-X|\mid X\subseteq Z\}$. In particular, $E$ is independent in the sum of $M_1$ and $M_2$ if and only if $r_{M_1}(X)+r_{M_2}(X)\geq |X|$ for every $X\subseteq E$.

A matroid is \emph{representable over some field $\bF$} if if there exists a family of vectors from a vector space over $\bF$ whose linear independence relation is the same as the independence relation of the matroid. The matroid is \emph{binary} if it is representable over $GF(2)$, and is \emph{regular} if it can be represented over any field. The following lemma gives a characterization of binary matroids in terms of circuits, see e.g.~\cite[Theorem~9.1.2]{oxley2011matroid}.

\begin{lem} \label{lem:bin}
A matroid is binary if and only if $C_1 \triangle C_2$ is a cycle for any cycles $C_1,C_2$.
\end{lem}

We will further rely on the following observation.

\begin{lem} \label{lem:0or2}
Let $T=\{t_1, t_2, t_3\}$ be a triangle of a binary matroid $M=(E,\cB(M))$ and let $F \subseteq E - T$. Then, $F+t_i\in\cB(M)$ for either none or exactly two of the indices $i\in\{1,2,3\}$.
\end{lem}
\begin{proof}
Assume first that at least two of the three sets form bases of $M$. We may assume that $F+t_1$ and $F+t_2$ are bases. Then, there exists a circuit $C \subseteq F+t_1 + t_2$ and necessarily $t_1, t_2\in C$. The set $C \triangle T$ is a cycle such that $t_3 \in C \triangle T \subseteq F+t_3$, hence $F+t_3$ is not a basis of $M$. This shows that at most two of the sets $F+t_1$, $F+t_2$ and $F+t_3$ are bases.

Suppose now that at most one of the three sets forms a basis. We may assume that $F+t_1$ and $F+t_2$ are not bases. If $F$ is not independent, then $F+t_3$ is clearly not a basis. Otherwise, let $C_1$ and $C_2$ be circuits such that $t_1 \in C_1 \subseteq F+t_1$ and $t_2 \in C_2 \subseteq F+t_2$. Then, $C_1 \triangle C_2 \triangle T$ is a cycle such $t_3 \in C_1 \triangle C_2 \triangle T \subseteq F+t_3$, hence $F+t_3$ is not a basis. This concludes the proof of the lemma.
\end{proof}

The matroid $R_{10}$ is a binary matroid that can be represented as the ten vectors in the five-dimensional vector space over $GF(2)$ that have exactly three nonzero entries, see Figure~\ref{fig:r10matrix}. The \emph{Fano matroid} $F_7$ is obtained from the Fano plane by calling a set independent if it contains at most two points or it has three points which are not lines of the plane, see Figure~\ref{fig:f7}. In other words, $F_7$ is the matroid with ground set $E=\{a,b,c,d,e,f,g\}$ whose bases are all subsets size of 3 except $\{a,b,d\}$, $\{b,c,e\}$, $\{a,c,f\}$, $\{a,e,g\}$, $\{c,d,g\}$, $\{b,f,g\}$ and $\{d,e,f\}$. 

\begin{figure}[t!]
\centering
\begin{subfigure}[t]{0.32\textwidth}
    \centering
    \includegraphics[width=0.85\textwidth]{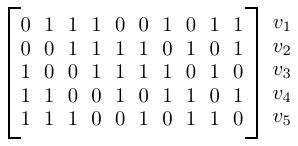}
    \caption{Representation of the binary matroid $R_{10}$ over $GF(2)$.}
    \label{fig:r10matrix}
\end{subfigure}\hfill
\begin{subfigure}[t]{0.32\textwidth}
    \centering
    \includegraphics[width=0.55\textwidth]{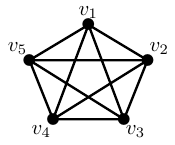}
    \caption{Representation of $R_{10}$ as an even-cycle matroid.}
    \label{fig:r10graph}
\end{subfigure}\hfill
\begin{subfigure}[t]{0.32\textwidth}
    \centering
    \includegraphics[width=0.55\textwidth]{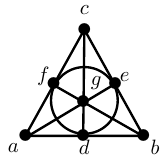}
    \caption{The bases of $F_7$ are the non-line 3-element sets of the Fano plane.}
    \label{fig:f7}
\end{subfigure}
\caption{Representations of $R_{10}$ and $F_7$.}
\label{fig:r10f7}
\end{figure}

\paragraph{Decomposition of Regular Matroids}

Let $M_1$ and $M_2$ be binary matroids on ground sets $E_1$ and $E_2$, respectively, such that $|E_1|, |E_2| < |E_1 \triangle E_2|$. Then, we denote by $M_1 \triangle M_2$ the binary matroid on ground set $E=E_1 \triangle E_2$ with cycles being the sets of the form $C_1 \triangle C_2$ where $C_i$ is a cycle of $M_i$ for $i=1,2$.

When $E_1 \cap E_2 = \emptyset$, then $M_1 \oplus_1 M_2 \coloneqq M_1 \triangle M_2$ is called the \emph{1-sum} or \emph{direct sum of $M_1$ and $M_2$}. Its family of bases is
\begin{align*}
    \cB(M_1 \oplus_1 M_2) = \{B_1 \cup B_2 \mid B_1 \in \cB(M_1), B_2 \in \cB(M_2)\} 
\end{align*}

When $|E_1 \cap E_2| = 1$, say $E_1\cap E_2 = \{t\}$, such that $t$ is nor a loop nor a coloop of $M_1$ or $M_2$, then $M_1 \oplus_2 M_2 \coloneqq M_1 \triangle M_2$ is called the \emph{2-sum of $M_1$ and $M_2$ along $t$}.
Its family of bases is
\begin{align*}
\cB(M_1 \oplus_2 M_2) {}={} & \{B'_1 \cup B_2 \mid B'_1 \in \cB(M_1/t), B_2 \in \cB(M_2\backslash t)\} \\ 
{}\cup{} & \{B_1 \cup B'_2 \mid B_1 \in \cB(M_1 \backslash t), B'_2 \in \cB(M_2 / t) \}.
\end{align*}

When $|E_1 \cap E_2| = 3$ and $E_1 \cap E_2 = T$ is a coindependent triangle of both $M_1$ and $M_2$, then $M_1 \oplus_3 M_2 \coloneqq M_1 \triangle M_2$ is called the \emph{3-sum of $M_1$ and $M_2$ along $T$}. The matroid $M_1 \oplus_3 M_2$ has rank $r(M_1)+r(M_2)-2$, and its family of bases is (see~\cite[Section 2.1]{aprile2022regular} together with Lemma~\ref{lem:0or2})
\begin{align*}
\cB(M_1 \oplus_3 M_2) {}={} & \{B''_1 \cup B_2 \mid B''_1 \in \cB(M_1 / T), B_2 \in \cB(M_2 \backslash T)\} \\
{}\cup{} &\{B'_1 \cup B'_2 \mid \exists i,j,k: \{i,j,k\} = \{1,2,3\}, B'_1+t_i, B'_1 + t_j \in \cB(M_1), B'_2 + t_i, B'_2 + t_k \in \cB(M_2)\} \\
{}\cup{} & \{B_1 \cup B''_2 \mid B_1 \in \cB(M_1 \backslash T), B''_2 \in \cB(M_2/T) \}.
\end{align*}

Seymour's fundamental decomposition theorem~\cite{seymour1980decomposition} gives a constructive characterization of regular matroids.

\begin{thm}[Seymour's decomposition theorem]\label{thm:seymour}
A matroid is regular if and only if it is obtained by means of 1-, 2- and 3-sums, starting from graphic and cographic matroids and copies of a certain 10-elements matroid $R_{10}$. 
\end{thm}

A binary matroid is said to be \emph{connected} or \emph{2-connected} if it is not a 1-sum, and \emph{3-connected} if it is not a 1-sum or a 2-sum of two matroids. While studying the extension complexity of the independence polytope of regular matroids, Aprile and Fiorini~\cite{aprile2022regular} recently gave a refinement of Seymour's result in the 3-connected case.

\begin{thm}[Aprile and Fiorini]\label{thm:af1}
Let $M$ be a 3-connected regular matroid distinct from $R_{10}$. There exists a tree $\cT$ such that each node $v\in V(\cT)$ is labeled with a graphic or cographic matroid $M_v$, each edge $uv\in E(\cT)$ has a corresponding 3-sum $M_u\oplus_3 M_v$, and $M$ is the matroid obtained by performing all the 3-sum  operations corresponding to the edges of $\cT$ in arbitrary order. Moreover, if $v\in V(\cT)$ is such that $M_v$ is the cographic matroid of a nonplanar graph $G_v$, then no nontrivial cut of $G_v$ is involved in any of the 3-sums.
\end{thm}

A tree $\cT$ satisfying the conditions of Theorem~\ref{thm:af1} is called a \emph{decomposition tree} of $M$, and the matroids corresponding to the nodes of $\cT$ are referred to as \emph{basic} matroids. It is worth mentioning that an analogous result was proved by Dinitz and Kortsarz in~\cite{dinitz2014matroid}, but their decomposition tree may involve 1- and 2-sums, and also 3-sums along nontrivial cuts.

To describe the dual of a 3-sum, we need the notion of \dy exchanges. In case of graphs, if $T$ is triangle of a graph $G$, then we perform a \dy exchange on $G$ by deleting the edges of $T$, adding a new vertex $v$ and edges new edges joining $v$ to vertices of $T$. More generally, consider a binary matroid $M$ and let $T$ be a coindependent triangle of $M$. Let $N$ be a matroid isomorphic to the graphic matroid of $K_4$ on ground set $T\cup T'$ where $T$ is a triangle of $N$, and the triad $T'$ of $N$ is disjoint from the ground set of $M$. We say that the matroid $\Delta_T(M) \coloneqq M \oplus_3 N$ is obtained from $M$ by performing a \emph{\dy exchange}. McGuinness~\cite{mcguinness2014base} gave a characterization of the dual of a 3-sum.

\begin{prop}[McGuinness]\label{prop:mcguinness}
Consider a 3-sum $M_1 \oplus_3 M_2$ along a coindependent triangle $T$ of $M_1$ and $M_2$. Then,
$(M_1 \oplus_3 M_2)^* = \Delta_T(M_1)^* \oplus_3 \Delta_T(M_2)^*$, where the \dy exchanges $\Delta_T(M_1)$ and $\Delta_T(M_2)$ are performed using the same matroid $N$ on ground set $T\cup T'$ and the 3-sum $\Delta_T(M_1)^*\oplus_3 \Delta_T(M_2)$ is performed using the common triangle $T'$ of $\Delta_T(M_1)^*$ and $\Delta_T(M_2)^*$.
\end{prop}

The reverse operation of a \dy exchange is called a \yd exchange. If $T$ is an independent triad of a binary matroid $M$, then we say that the matroid $\nabla_T(M) \coloneqq \Delta_T(M^*)^*$ is obtained from $M$ by performing a \yd exchange. The \dy and \yd exchanges are indeed reverse operations of each other, see \cite[Proposition~11.5.11]{oxley2011matroid}. If $u$ is a degree 3 vertex of a graph $G$ with distinct adjacent vertices $x$, $y$ and $z$, then we can perform the \yd operation on the graphic matroid $M(G)$ by deleting $u$ and adding the edges $xy$, $yz$ and $xz$. In particular, if $T$ is an independent triad of a graphic matroid $M(G)$ corresponding to the edges adjacent to a trivial cut $G$, then $\nabla_T(M(G))$ is a graphic matroid. This implies the following.

\begin{lem}\label{lem:co}
    If $T$ is a coindependent triangle of a cographic matroid $M^*(G)$ corresponding to a trivial cut of $G$, then $\Delta_T(M^*(G))$ is cographic.
\end{lem}

\begin{rem}
We note that if $T$ does not correspond to a trivial cut of $G$, then $\Delta_T(M^*(G))$ might not be a cographic matroid. As an example, if $e$ is an edge of $K_5$ and $T$ is the triangle of $K_5$ formed by the vertices not adjacent to $e$, then, $\Delta_T(M(K_5-e)) = M(K_{3,3})$, see also~\cite[Figure~11.20]{oxley2011matroid}. Since $K_5-e$ is a planar and $K_{3,3}$ is a nonplanar graph, $M(K_5-e)$ is a cographic matroid while $M(K_{3,3})$ is not.
\end{rem}

\paragraph{Algorithms and Oracles}

In matroid algorithms, it is usually assumed that the matroid is given by an \emph{oracle} and the running time is measured by the number of oracle calls and other conventional elementary steps. There are many different types of oracles that are often used, the independence, circuit and rank oracles probably being the most standard ones. For a matroid $M=(E,\cI)$ and set $X\subseteq E$ as an input, an independence oracle answers \emph{``Yes''} if $X$ is independent and \emph{``No''} otherwise, a circuit oracle answers \emph{``Yes''} if $X$ is a circuit and \emph{``No''} otherwise, and a rank oracle gives back $r_M(X)$. 

In fact, these oracles have the same computational power. An oracle $\cO_1$ is \emph{polynomially reducible} to another oracle $\cO_2$ if $\cO_1$ can be implemented by using a polynomial number of oracle calls to $\cO_2$ measured in terms of the size of the ground set. Two oracles are \emph{polynomially equivalent} if they are mutually polynomially reducible to each other. It is not difficult to show that the independence, circuit and rank oracles are polynomially equivalent, see e.g.~\cite{robinson1980computational}. 

Let $E$ denote the ground set of $M$, and $X$ and $Y$ be disjoint subsets of $E$. Then the rank function of the minor $M/X\backslash Y$ is $r_{M/X\backslash Y}(Z)=r(Z\cup X)-r(Z)$ for $Z\subseteq E-(X\cup Y)$, and the rank function of the dual $M^*$ is $r_{M^*}(Z)=|Z|-(r_M(E)-r_M(E-Z))$ for $Z\subseteq E$, see e.g.~\cite{oxley2011matroid}. That is, given an independence oracle access to the matroid $M$, independence oracles can be implemented for any minor and the dual of $M$ by the polynomial equivalence of the rank and independence oracles. Therefore, we will use these basic matroid operations in our algorithm.

For a binary matroid $M$, it can be decided if $M$ is not connected, connected but not $3$-connected, or $3$-connected using a polynomial number of oracle calls~\cite[Theorem 8.4.1]{truemper1990decomposition}. Moreover, the algorithm also provides a 1-sum decomposition in the first case, a 2-sum decomposition in the second case, and a 3-sum decomposition if it exists in the third case. When applied to a regular matroid recursively, the algorithm eventually gives a decomposition $M$ into basic matroids each of which is either graphic, cographic or isomorphic to $R_{10}$. If the matroid is $3$-connected and is not $R_{10}$, then~\cite{aprile2022regular} describes an algorithm how to modify this decomposition until it gives a decomposition tree with no bad nodes using a polynomial number of oracle calls. These observations together imply that for any $3$-connected matroid different from $R_{10}$, we can efficiently determine a decomposition tree not containing bad nodes.

\section{Reduction to 3-Connected Case Without Small Cocircuits}
\label{sec:reductions}

In this section, we focus on how the exchange distance behaves for basic matroid operations such as contraction and taking 1- or 2-sums. Furthermore, we identify structural properties of regular matroids such as the existence of a \emph{tight set} or a \emph{triad} that allow for reduction in the problem size. As mentioned in the introduction, these reduction steps will eventually make it possible to write up the matroid as the 3-sum of a regular matroid and the graphic matroid of a 4-regular graph. Algorithmic aspects of the preprocessing steps are discussed at the end of the section.

Let $M$ be the 1-, 2- or 3-sum of binary matroids $M_\circ=(E_\circ,\cB(M_\circ))$ and $M_\bullet=(E_\bullet,\cB(M_\bullet))$ along $T$ where $T$ is empty in case of 1-sums, it consists of a single element in case of 2-sums and of three elements in case of 3-sums. For any set $X\subseteq E_\circ\cup E_\bullet$, we define $X^\circ\coloneqq X\cap (E_\circ-T)$ and $X^\bullet\coloneqq X\cap (E_\bullet-T)$. In particular, for any basis $B\in\cB(M)$, we have $B^\circ= B\cap E_\circ$ and $B^\bullet= B\cap E_\bullet$. Note that for 2- and 3-sums, $B^\circ$ and $B^\bullet$ are not necessarily bases of $M_\circ$ and $M_\bullet$, respectively; see the characterization of bases in Section~\ref{sec:prelim}.

Since we will prove Theorem~\ref{thm:regwhite} and Theorem~\ref{thm:reggabow} in a stronger form for graphic matroids, we formulate some of the reductions for $F$-avoiding exchange sequences whose last step is partially fixed. This results in a series of rather technical lemmas, but this should not deter the interested reader from the later sections. We encourage first-time readers to skip these technical parts and only return to them after getting a general understanding of the structure of the proof. 

Operations similar to those discussed in Section~\ref{sec:disjoint} and Section~\ref{sec:tight} were implicitly mentioned in~\cite{white1980unique}, while an operation similar to the one discussed in Section~\ref{sec:2sum} was considered in~\cite{shibata2016toric}. However, we will deduce stronger properties of the reduction steps and also discuss the algorithmic aspects. 

\subsection{Making the Bases Disjoint}
\label{sec:disjoint}

We start with the simple observation that it suffices to consider compatible pairs consisting of disjoint bases. 

\begin{lem}\label{lem:disjoint}
Let $\cX=(X_1,X_2)$ and $\cY=(Y_1,Y_2)$ be compatible pairs of bases of a matroid $M$ and let $F\subseteq (X_1\cap Y_1)\cup (X_2\cap Y_2)$. Define $\cX'\coloneqq(X_1-X_2,X_2-X_1)$, $\cY'\coloneqq(Y_1-Y_2,Y_2-Y_1)$ and $F'\coloneqq F-(X_1\cap X_2)$. If there exists an $F'$-avoiding $\cX'$-$\cY'$ exchange sequence in $M/(X_1\cap X_2)$ of width $w$ and length $\ell$, then there exists an $F$-avoiding $\cX$-$\cY$ exchange sequence in $M$ of width $w$ and length $\ell$. Furthermore, if $h\in E-(X_1\cap X_2)$ is used in the last step of the $\cX'$-$\cY'$ exchange sequence, then it can be assumed to be used in the last step of the $\cX$-$\cY$ exchange sequence as well.

\end{lem}
\begin{proof}
Recall that $(X_1,X_2)$ and $(Y_1,Y_2)$ are compatible if $X_1\cap X_2=Y_1\cap Y_2$ and $X_1\cup X_2=Y_1\cup Y_2$. This implies that $\cX'$ and $\cY'$ form compatible basis pairs of $M/(X_1\cap X_2)$. As any sequence of symmetric exchanges that transforms $\cX'$ into $\cY'$ also transforms $\cX$ into $\cY$, the lemma follows.
\end{proof}

By the lemma, it suffices to consider instances where $X_1\cap X_2=Y_1\cap Y_2=\emptyset$. Furthermore, since the elements not contained in any of the bases cannot participate in exchanges and hence can be deleted, we can assume without loss of generality that $E=X_1\cup X_2=Y_1\cup Y_2$ holds. 

\subsection{Excluding Tight Sets}
\label{sec:tight}

Given a matroid $M$ over ground set $E$, a set $Z\subseteq E$ is called \emph{tight} if $|Z|=2\cdot r_M(Z)$. A tight set $Z$ is called \emph{nontrivial} if $\emptyset\neq Z\subsetneq E$. Nontrivial tight sets are special for the following reason: every partition $E=X_1\cup X_2$ into two disjoint bases necessarily satisfies $|X_i\cap Z|=r_M(Z)$ for $i=1,2$. In other words, pairs of disjoint bases of $M$ are exactly the pairs of disjoint bases of the matroid $M|Z\oplus_1 M/Z$. This observation allows us to reduce the size of the problem along a nontrivial tight set.

\begin{lem}\label{lem:tight}
Let $\cX=(X_1,X_2)$ and $\cY=(Y_1,Y_2)$ be compatible pairs of disjoint bases of a matroid $M$, $F\subseteq (X_1\cap Y_1)\cup (X_2\cap Y_2)$, and let $\emptyset\neq Z\subsetneq X_1\cup X_2$ be a tight set. Define $F'\coloneqq F\cap Z$, $F''\coloneqq F-Z$, $\cX'\coloneqq(X_1\cap Z,X_2\cap Z)$, $\cX''\coloneqq(X_1-Z,X_2-Z)$, $\cY'\coloneqq(Y_1\cap Z,Y_2\cap Z)$ and $\cY''\coloneqq(Y_1-Z,Y_2-Z)$. If there exists an $F'$-avoiding $\cX'$-$\cY'$ exchange sequence in $M|Z$ of width $w'$ and length $\ell'$ and an $F''$-avoiding $\cX''$-$\cY''$ exchange sequence in $M/Z$ of width $w''$ and length $\ell''$, then there exists an $F$-avoiding $\cX$-$\cY$ exchange sequence in $M$ of width $\max\{w', w''\}$ and length $\ell'+\ell''$. Furthermore, if $h\in E$ is used in the last step of the $\cX''$-$\cY''$ exchange sequence, then it can be assumed to be used in the last step of the $\cX$-$\cY$ exchange sequence as well.
\end{lem}
\begin{proof}
By the definition of contraction, the concatenation of the two exchange sequences results in an $\cX$-$\cY$ exchange sequence with the properties stated.
\end{proof}

If $M$ is the 1-sum of matroids $M_\circ=(E_\circ,\cB(M_\circ))$ and $M_\bullet=(E_\bullet,\cB(M_\bullet))$, then $M|E_\circ=M_\circ$ and $M/E_\circ=M_\bullet$. Furthermore, the bases of $M$ are exactly the unions of a basis of $M_\circ$ and a basis of $M_\bullet$. Hence, for any pair $(X_1,X_2)$ of disjoint bases of $M$, the set $X^\circ_1\cup X^\circ_2$ is tight since $|X^\circ_1\cup X^\circ_2|=2\cdot r_{M_\circ}(X^\circ_1\cup X^\circ_2)=2\cdot r_M(X^\circ_1\cup X^\circ_2)$. Therefore, Lemma~\ref{lem:tight} implies the following. 

\begin{cor} \label{cor:1sum}
Let $\cX=(X_1,X_2)$ and $\cY=(Y_1,Y_2)$ be compatible pairs of disjoint bases of a matroid $M=M_\circ\oplus_1 M_\bullet$. 
Define $\cX' \coloneqq (X^\circ_1, X^\circ_2)$, $\cX''\coloneqq (X^\bullet_1, X^\bullet_2)$, $\cY' \coloneqq (Y^\circ_1, Y^\circ_2)$ and $\cY'' \coloneqq (Y^\bullet_1, Y^\bullet_2)$. If there exists an $\cX'$-$\cY'$ exchange sequence in $M_\circ$ of width $w'$ and length $\ell'$ and an $\cX''$-$\cY''$ exchange sequence in $M_\bullet$ of width $w''$ and length $\ell''$, then there exists an $\cX$-$\cY$ exchange sequence in $M$ of width $\max\{w', w''\}$ and length $\ell'+\ell''$.
\end{cor}

\subsection{Reduction to 3-Connected Matroids}
\label{sec:2sum}

When the matroid happens to be the 2-sum of matroids, the problem admits a reduction similar to the one used for tight sets. However, while merging the solutions to the subproblems was trivial for tight sets, it becomes much more involved for 2-sums. To get a better understanding of this difficulty, let $X_1$ and $X_2$ be disjoint bases of $M=M_\circ\oplus_2 M_\bullet$ where the 2-sum is along an element $t$. Assume that $X^\circ_1\in\cB(M_\circ\backslash t)$ and $X^\circ_2\in \cB(M_\circ/ t)$. This implies that $X^\circ_1+t$ and $X^\circ_2$ are bases of $M_\circ$, and that $X^\bullet_1\in\cB(M_\bullet/ t)$ and $X^\bullet_2\in\cB(M_\bullet\backslash t)$ by the definition of 2-sums. Consider a symmetric exchange  $X^\circ_1+f,X^\circ_2-f+t$ between $X^\circ_1+t$ and $X^\circ_2$ in $M_\circ$. Then, unfortunately, this step does not correspond to a feasible symmetric exchange between $X_1$ and $X_2$ in $M$, since both $X^\circ_1+f\in\cB(M_\circ/t)$ and $X^\bullet_1\in\cB(M_\bullet/ t)$.

The main result of this section is to show that the exchanges can be scheduled on the two sides of the 2-sum in a way that avoids the problem described above. The next lemma, when used in conjunction with Lemma~\ref{lem:tight}, eventually reduces the problem to the case of 3-connected matroids.

\begin{lem}\label{lem:2sum}
Let $\cX=(X_1,X_2)$ and $\cY=(Y_1,Y_2)$ be compatible pairs of disjoint bases of a matroid $M=M_\circ\oplus_2 M_\bullet$ where the 2-sum is along an element $t$. For $i=1,2$, set $X'_i\coloneqq X^\circ_i,X''_i\coloneqq X^\bullet_i+t$ if $X^\circ_i\in\cB(M_\circ\backslash t)$ and $X'_i\coloneqq X^\circ_i+t,X''_i\coloneqq X^\bullet_i$ otherwise, and $Y'_i\coloneqq Y^\circ_i, Y''_i\coloneqq Y^\bullet_i+t$ if $Y^\circ_i\in\cB(M_\circ\backslash t)$ and $Y'_i\coloneqq Y^\circ_i+t, Y''_i\coloneqq Y^\bullet_i$ otherwise. Define $\cX'\coloneqq(X'_1,X'_2)$, $\cX''\coloneqq(X''_1,X''_2)$, $\cY'\coloneqq(Y'_1,Y'_2)$ and $\cY''\coloneqq(Y''_1,Y''_2)$. If there exists an $\cX'$-$\cY'$ exchange sequence in $M_\circ$ of width $w'$ and length $\ell'$ and an $\cX''$-$\cY''$ exchange sequence in $M_\bullet$ of width $w''$ and length $\ell''$, then there exists an $\cX$-$\cY$ exchange sequence in $M$ of width at most $w'+w''$ and length at most $\ell'+\ell''-1$ if both exchange sequences involve $t$ and $\ell'+\ell''$ otherwise.
\end{lem}
\begin{proof}
Using the description of $\cB(M_\circ \oplus_2 M_\bullet)$, we may assume that $X^\circ_1 \in \cB(M_\circ\backslash t)$ and $X^\bullet_1 \in \cB(M_\bullet / t)$. 
If $X^\circ_2 \in \cB(M_\circ \backslash t)$, then $X^\circ_1 \cup X^\circ_2$ is a tight set in $M$ and the statement follows from Lemma~\ref{lem:tight}. Otherwise, $X^\circ_2 \in \cB(M_\circ / t)$ and thus $X^\bullet_2 \in \cB(M_\bullet \backslash t)$. 
We prove the lemma in two steps.

First, consider the case when the $\cX'$-$\cY'$ exchange sequence in $M_\circ$ does not involve the element $t$.
Note that in this case $\cX'=(X^\circ_1, X^\circ_2+t)$, $\cX''=(X^\bullet_1+t, X^\bullet_2)$, $\cY'=(Y^\circ_1, Y^\circ_2+t)$ and $\cY''=(Y^\bullet_1+t, Y^\bullet_2)$.
We construct an $\cX$-$\cY$ exchange sequence as follows. We start with the steps of the $\cX'$-$\cY'$ exchange sequence, which transform $\cX=(X_1,X_2)$ into the basis pair $(Y^\circ_1 \cup X^\bullet_1, Y^\circ_2 \cup X^\bullet_2)$.
By the symmetric exchange axiom, there exists $e\in Y^\circ_1-(Y^\circ_2+t)$ such that $Y^\circ_1-e+t,Y^\circ_2+e\in\cB(M_\circ)$.
From this point, we perform the steps of the $\cX''$-$\cY''$ exchange sequence, but whenever a symmetric exchange uses $t$ and some other element $f$, then exchange $e$ and $f$ instead.
Formally, if a symmetric exchange transforms $(Z^\bullet_1+t, Z^\bullet_2)$ into $(Z^\bullet_1+f, Z^\bullet_2-f+t)$, then this is replaced by the symmetric exchange that transforms $(Y^\circ_1 \cup Z^\bullet_1, Y^\circ_2 \cup Z^\bullet_2)$ into $((Y^\circ_1-e)\cup (Z^\bullet_1+f), ((Y^\circ_2+e)\cup (Z^\bullet_2-f))$ in $M$.
Similarly, if a symmetric exchange transforms $(Z^\bullet_1, Z^\bullet_2+t)$ into $(Z^\bullet_1-f+t, Z^\bullet_2+f)$, then this step is replaced by the symmetric exchange that transforms $((Y^\circ_1-e) \cup Z^\bullet_1, (Y^\circ_2+e) \cup Z^\bullet_2)$ into $(Y^\circ_1\cup (Z^\bullet_1-f), Y^\circ_2 \cup (Z^\bullet_2+f))$.
In both cases, the pair obtained consists of disjoint bases of $M$ due to the choice of $e$. It is not difficult to check that at the end of the procedure, we arrive at the basis pair $(Y_1,Y_2)$. The $\cX$-$\cY$ exchange sequence thus obtained has width at most $w'+w''$ and length $\ell'+\ell''$.

By symmetry, it remains to consider the case when both the $\cX'$-$\cY'$ exchange sequence in $M_\circ$ and the ${\cX''\text{-}\cY''}$ exchange sequence in $M_\bullet$ use the element $t$ at least once. Let $m'$ and $m''$ denote the number of occurrences of $t$ in these sequences; we may assume that $m'<m''$. We construct an $\cX$-$\cY$ exchange sequence as follows. We perform the steps of both the $\cX'$-$\cY'$ and $\cX''$-$\cY''$ exchange sequences, but we align the exchanges involving $t$ on both sides, see Figure~\ref{fig:2sum}. Formally, we always perform the steps of the $\cX'$-$\cY'$ exchange sequence until we reach the next step that involves $t$, say, transforms a basis pair $(Z^\circ_1, Z^\circ_2+t)$ into $(Z^\circ_1-e+t, Z^\circ_2+e)$.  
From this point, we perform the steps of the $\cX''$-$\cY''$ exchange sequence until we reach the next step that involves $t$, say, transforms $(Z^\bullet_1+t, Z^\bullet_2)$ into $(Z^\bullet_1+f, Z^\bullet_2-f+t)$. 
Then these two steps are replaced by the symmetric exchange that transforms 
$(Z^\circ_1 \cup Z^\bullet_1, Z^\circ_2 \cup Z^\bullet_2)$ into $((Z^\circ_1-e) \cup (Z^\bullet_1+f), (Z^\circ_2+e)\cup (Z^\bullet_2-f))$ in $M$.
Once there are no more steps using $t$ on the side of $M_\circ$, the exchange sequence can be finished as discussed in the previous case. The $\cX$-$\cY$ exchange sequence thus obtained has width at most $w'+w''$ and length $\ell'+\ell''-m'\leq \ell'+\ell''-1$.
\end{proof}

\begin{figure}[t!]
\centering
\begin{subfigure}[t]{0.27\textwidth}
    \centering
    \includegraphics[width=0.7\textwidth]{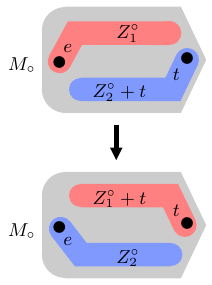}
    \caption{A step of the $\cX'$-$\cY'$ exchange sequence transforming $(Z^\circ_1\!,\!Z^\circ_2\!+t)$ into $(Z^\circ_1\!-\!e\!+\!t,\! Z^\circ_2\!+\!e)$.}
    \label{fig:2sum1}
\end{subfigure}\hfill
\begin{subfigure}[t]{0.27\textwidth}
    \centering
    \includegraphics[width=0.7\textwidth]{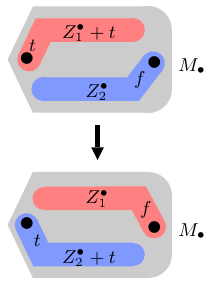}
    \caption{A step of the $\cX''$-$\cY''$ exchange sequence transforming $(Z^\bullet_1\!+t,\!Z^\bullet_2)$ into $(Z^\bullet_1\!+\!f\!,\!Z^\bullet_2\!-\!f\!+\!t)$.}
    \label{fig:2sum2}
\end{subfigure}\hfill
\begin{subfigure}[t]{0.4\textwidth}
    \centering
    \includegraphics[width=0.8\textwidth]{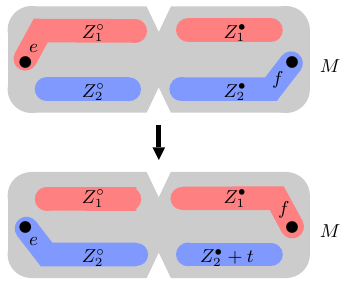}
    \caption{The corresponding step of the $\cX$-$\cY$ exchange sequence transforming $(Z^\circ_1\! \cup\! Z^\bullet_1, Z^\circ_2\! \cup\! Z^\bullet_2)$ into $((Z^\circ_1\!-\!e)\! \cup\! (Z^\bullet_1\!+\!f), (Z^\circ_2\!+\!e)\!\cup\! (Z^\bullet_2\!-\!f))$.}
    \label{fig:2sum3}
\end{subfigure}
\caption{Illustration of Lemma~\ref{lem:2sum}, where steps of the $\cX'$-$\cY'$ and $\cX''$-$\cY''$ exchange sequences involving $t$ are combined to obtain a step of the $\cX$-$\cY$ exchange sequence.}
\label{fig:2sum}
\end{figure}

\subsection{Excluding Cocircuits of Size Three}
\label{sec:cocircuit}

Recall that a triad is a cocircuit of size three. Let $M$ be a matroid, $\cX=(X_1,X_2)$ and $\cY=(Y_1,Y_2)$ be compatible pairs of disjoint bases of $M$, and $T=\{t_1,t_2,t_3\}$ be a triad of $M$ such that $T\subseteq X_1\cup X_2=Y_1\cup Y_2$. The pairs $\cX$ and $\cY$ are called \emph{consistent on $T$} if $X_1 \cap T = Y_1 \cap T$ or $X_1 \cap T = Y_2 \cap T$, that is, the bases in $\cX$ partition the elements of $T$ the same way as the bases in $\cY$. We will use the following simple technical claim.

\begin{cl}\label{cl:two}
Let $\cX=(X_1,X_2)$ be a pair of disjoint bases of a matroid $M$ and $T=\{t_1,t_2,t_3\}\subseteq X_1\cup X_2$ be a triad of $M$ such that $|X_1\cap T|=2$. Then, $((X_1-T)+\{t_i,t_j\},(X_2-T)+t_k)$ forms a pair of disjoint bases for at least two choices of indices satisfying $\{i,j,k\}=\{1,2,3\}$.
\end{cl}
\begin{proof}
Without loss of generality, we may assume that $t_1,t_2\in X_1$. Since $X_1$ is a basis of $M$, $X_1+t_3$ contains a unique circuit $C$. By Lemma~\ref{lem:int}, the intersection of $C$ and $T$ has size different from one, hence $C\cap\{t_1,t_2\}\neq\emptyset$. We may assume that $t_1\in C$, implying $X_1-t_1+t_3$ being a basis. It remains to show that $X_2+t_1-t_3$ is also a basis. Suppose to the contrary that this does not hold, that is, $X_2+t_1-t_3$ contains a circuit $C'$. Then, by $X_2\cap T=\{t_3\}$, we get $C'\cap T=\{t_1\}$, contradicting Lemma~\ref{lem:int}.
\end{proof}

First, we show that if the basis pairs $\cX,\cY$ are not consistent on a triad $T$, then one can obtain another pair $\cX',\cY'$ that are consistent on $T$ at the cost of at most two symmetric exchanges. 

\begin{lem}\label{lem:cocircuit1}
Let $\cX=(X_1,X_2)$ and $\cY=(Y_1,Y_2)$ be compatible pairs of disjoint bases of a matroid $M$ that are not consistent on a triad $T=\{t_1,t_2,t_3\}\subseteq X_1\cup X_2$. Then, there exists compatible pairs of disjoint bases $\cX'=(X'_1,X'_2)$ and $\cY'=(Y'_1,Y'_2)$ that are consistent on $T$ and are obtained by applying at most one symmetric exchange to $\cX$ and to $\cY$, respectively.  
\end{lem}
\begin{proof}
Define $\cP_1\coloneqq\{((X_1-T)+\{t_i,t_j\},(X_2-T)+t_k)\mid\{i,j,k\}=\{1,2,3\}\}$ if $|X_1\cap T|=2$ and $\cP_1\coloneqq\{((X_1-T)+t_i,(X_2-T)+\{t_j,t_k\})\mid\{i,j,k\}=\{1,2,3\}\}$ otherwise. Observe that each member of $\cP_1$ can be obtained from $\cX$ by exchanging at most one pair of elements; however, this might not be a feasible symmetric exchange. Similarly, define $\cP_2\coloneqq\{ (( Y_1-T)+\{t_i,t_j\},(Y_2-T)+t_k)\mid\{i,j,k\}=\{1,2,3\}\}$ if $|Y_1\cap T|=2$ and $\cP_2\coloneqq\{((Y_1-T)+t_i,(Y_2-T)+\{t_j,t_k\})\mid\{i,j,k\}=\{1,2,3\}\}$ otherwise. Observe that each member of $\cP_2$ can be obtained from $\cY$ by exchanging at most one pair of elements; again, this might not be a feasible symmetric exchange.

By Claim~\ref{cl:two}, at least two members of $\cP_1$ and at least two members of $\cP_2$ consist of disjoint bases. Therefore, there exist $\cX'\in\cP_1$ and $\cY'\in\cP_2$ that are consistent on $T$, concluding the proof of the lemma.
\end{proof}

Once the basis pairs are consistent on a triad, the problem size can be decreased by contracting and deleting appropriate elements of the triad.  

\begin{lem}\label{lem:cocircuit2}
Let $\cX=(X_1,X_2)$ and $\cY=(Y_1,Y_2)$ be compatible pairs of disjoint bases of a matroid $M$ that are consistent on a triad $T=\{t_1,t_2,t_3\}\subseteq X_1\cup X_2$  where $t_1,t_2\in X_1$, and let $F\subseteq ((X_1\cap Y_1)\cup(X_2\cap Y_2))-T$. Define $\cX'\coloneqq(X_1-t_2,X_2-t_3)$, and set $\cY'\coloneqq (Y_1-t_2,Y_2-t_3)$ if $t_1,t_2\in Y_1$ and $\cY'\coloneqq (Y_1-t_3,Y_2-t_2)$ otherwise. If there exists an $F$-avoiding $\cX'$-$\cY'$ exchange sequence in $M/t_2\backslash t_3$ of width $w$ and length $\ell$, then there exists an $F$-avoiding $\cX$-$\cY$ exchange sequence in $M$ of width $w$ and length at most $\ell+ w$. Furthermore, if $h\in E-(T\cup F)$ is used in the last step of the $\cX'$-$\cY'$ exchange sequence, then it can be assumed to be used in the last step of the $\cX$-$\cY$ exchange sequence as well. 
\end{lem}
\begin{proof}
For any pair of disjoint bases $\cZ=(Z_1,Z_2)$ of $M/t_2\backslash t_3$, we denote by $\cZ^+=(Z_1,Z_2)^+=(Z^+_1,Z^+_2)$ the pair where, for $i=1,2$, $Z^+_i\coloneqq Z_i+t_2$ if $t_1\in Z_i$ and $Z^+_i\coloneqq Z_i+t_3$ otherwise. Observe that $\cZ^+$ is a pair of disjoint bases of $M$. Indeed, $Z^+_i=Z_i+t_2$ is a basis of $M$ by the definition of contraction. If $t_1\notin Z_i$, then $Z_i+\{t_2,t_3\}$ contains a unique circuit $C$ that contains $t_3$. By Lemma~\ref{lem:int}, the intersection of $C$ and $T$ cannot have size one, hence $t_2\in C$ as well, showing that $Z_i+t_3$ is a basis of $M$.

Fix an $F$-avoiding $\cX'$-$\cY'$ exchange sequence in $M/t_2\backslash t_3$ of length $\ell$ and width $w$. The idea is to add certain extra steps to obtain a solution to the original instance. Consider a symmetric exchange in the sequence that transforms $(Z_1,Z_2)$ into $(Z_1-e+f,Z_2-f+e)$. Without loss of generality, we may assume that $t_1\in Z_1$. If $e$ is distinct from $t_1$, then $(Z_1,Z_2)^+=(Z_1+t_2,Z_2+t_3)$ and $(Z_1-e+f,Z_2-f+e)^+=(Z_1-e+\{t_2,f\},Z_2-f+\{t_3,e\})$, hence these pairs also differ in a single symmetric exchange in $M$. However, if $e=t_1$ then $(Z_1,Z_2)^+=(Z_1+t_2,Z_2+t_3)$ and $(Z_1-t_1+f,Z_2-f+t_1)^+=(Z_1-t_1+\{t_3,f\},Z_2-f+\{t_1,t_2\})$, and these pairs cannot be obtained from each other by a single symmetric exchange. In this case, consider the pairs $(Z_1+t_3,Z_2+t_2)$ and $(Z_1-t_1+\{t_2,t_3\},Z_2+t_1)$. By Claim~\ref{cl:two}, at least one of these pairs consists of disjoint bases of $M$. Furthermore, any of them can be obtained from both $(Z_1,Z_2)^+$ and $(Z_1-t_1+f,Z_2-f+t_1)^+$ by using a single symmetric exchange. 

Summarizing the above, a symmetric exchange of elements $e$ and $f$ in the $\cX'$-$\cY'$ exchange sequence is left unchanged if $e,f\neq t_1$. Otherwise, if, say, $e=t_1$, it is replaced by two steps: the first exchanging $t_1$ and $t_i$ and the second exchanging $t_j$ and $f$ for some appropriate choice of $i$ and $j$ satisfying $\{i,j\}=\{2,3\}$. Observe that these modifications do not increase the usage of an element in $E-\{t_2,t_3\}$, hence the width of the new sequence is also $w$. Furthermore, the length of the sequence increases by the number of symmetric exchanges involving $t_1$, hence the length of the new sequence is at most $\ell+w$. Finally, note that the new sequence is $F$-avoiding as well, and its last step uses the elements of $E-(T\cup F)$ that were involved in the last step of the $\cX'$-$\cY'$ exchange sequence, thus concluding the proof of the lemma.
\end{proof}

The two lemmas allow us to reduce the problem size if the matroid contains a triad. Indeed, the basis pairs can be made consistent on any triad with the help of Lemma~\ref{lem:cocircuit1}, which requires at most two symmetric exchanges. Once the basis pairs are consistent on a triad, we can decrease the number of elements as in Lemma~\ref{lem:cocircuit2}. If the exchange sequence in the reduced instance has width $w$ and length $\ell$, then we get an exchange sequence of width $w+2$ and length $\ell+w+2$ for the original instance. 

\subsection{Algorithmic Aspects}
\label{sec:algored}

The preprocessing steps discussed in the previous subsections do not only reduce the problem size in a theoretical sense, but are also algorithmically tractable if the matroid $M$ is given by an independence oracle. Recall that a 1-sum or 2-sum decomposition of $M$ can be determined, if exists, efficiently. Thus it suffices to show that one can find a triad or a tight set of a matroid using a polynomial number of oracle calls. 

By definition, a triad is a cocircuit of size three, or equivalently, a circuit of size three of the dual matroid. Since an independence oracle of the dual matroid can be implemented using the independence oracle of $M$, the existence of such a circuit can be decided by checking every 3-elements subset of the ground set.

Assume now that the ground set of $M$ is the disjoint union of two bases, say $X_1$ and $X_2$. Then for any set $Z$, we have $2\cdot r_M(Z)\geq |X_1\cap Z|+|X_2\cap Z|=|Z|$, and equality holds if and only if $Z$ is tight. Hence to decide whether $X_1\cup X_2$ properly contains a nonempty tight set of $M$, it suffices to minimize the submodular function $f(Z)\coloneqq r_M(Z)-|Z|/2$ over the sets $\emptyset\neq Z\subsetneq X_1\cup X_2$, which can be performed in strongly polynomial time if given access to the independence oracle~\cite{cunningham1984testing}.

Finally, we show that the width and length bounds of Lemma~\ref{lem:disjoint}, Lemma~\ref{lem:tight}, Lemma~\ref{lem:2sum}, Lemma~\ref{lem:cocircuit1} and Lemma~\ref{lem:cocircuit2} are consistent with the statement of Theorem~\ref{thm:regwhite}. Before that, we need the following simple observation.

\begin{cl}\label{cl:triv}
Let $\cX=(X_1,X_2)$ and $\cY=(Y_1,Y_2)$ be compatible basis pairs of a matroid $M$ of rank $r\leq 2$, $F\subseteq (X_1\cap Y_1)\cup (X_2\cap Y_2)$. Then there exists an $F$-avoiding $\cX$-$\cY$ exchange sequence of width at most $1$ and length at most $r$. Furthermore, if $h\in (X_1\cup X_2)-F$, then the last step of the sequence can be assumed to use $h$.
\end{cl}
\begin{proof}
The claim is straightforward to check for matroids of rank at most two.
\end{proof}

With the help of the claim, we are now ready to prove that the inverse operations of the reduction steps preserve the quadratic running time. Since 2-sum behaves differently for $F$-avoiding exchange sequences than the other operations, we do this in the form of two corollaries. Moreover, we state the corollaries parameterized by a constant $c\geq 1$; the reason is that we will choose $c$ to be $1$ for graphic matroids and $2$ for general regular matroids. For nondisjoint bases, tight sets and triads, we get the following.

\begin{cor}\label{cor:numstepstriad}
Let $\cX=(X_1,X_2)$ and $\cY=(Y_1,Y_2)$ be compatible pairs of bases of a matroid $M=(E,\cI)$ of rank $r\geq 3$ and let $F\subseteq (X_1\cap Y_1)\cup (X_2\cap Y_2)$. Assume that for any minor $M'=(E',\cI')$ of $M$ and for any pair $\cX',\cY'$ of compatible pairs of disjoint bases of $M'$, there exists an $F'$-avoiding $\cX'$-$\cY'$ exchange sequence in $M'$ of width at most $2\cdot c\cdot (r'-1)$ and length at most $c\cdot r'^{\,2}$, where $F'=F\cap E'$, $2\leq r'<r$ is the rank of $M'$ and $c\geq 1$. If either $X_1\cap X_2\neq\emptyset$, $M$ has a tight set $\emptyset\neq Z\subsetneq X_1\cup X_2$, or $M$ has a triad $T\subseteq X_1\cup X_2$, then there exists an $F$-avoiding $\cX$-$\cY$ exchange sequence in $M$ of width at most $2\cdot c\cdot (r-1)$ and length at most $c\cdot r^2$. 
\end{cor}
\begin{proof}
If $X_1\cap X_2\neq\emptyset$, then let $r'$ denote the rank of $M/(X_1\cap X_2)$. Note that $r'<r$ holds. Our assumption, Claim~\ref{cl:triv}, and Lemma~\ref{lem:disjoint} then imply the existence of an $F$-avoiding $\cX$-$\cY$ exchange sequence of width at most $\max\{1,2\cdot c\cdot (r'-1)\}<2\cdot c\cdot (r-1) $ and length at most $\max\{1,c\cdot r'^{\,2}\}<c\cdot r^2$. 

Otherwise, $X_1\cap X_2=Y_1\cap Y_2=\emptyset$. If $\emptyset\neq Z\subsetneq X_1\cup X_2$ is a tight set, then let $r'$ and $r''$ denote the ranks of $M|Z$ and $M/Z$, respectively. Note that $r=r'+r''$ holds. Our assumption, Claim~\ref{cl:triv}, and Lemma~\ref{lem:tight} then imply the existence of an $F$-avoiding $\cX$-$\cY$ exchange sequence of width at most $\max\left\{\max\{1,2\cdot c\cdot (r'-1)\},\max\{1,2\cdot c\cdot(r''-1)\}\right\}<2\cdot c\cdot (r-1)$ and length at most $\max\{1,c\cdot r'^{\,2}\}+\max\{1,c\cdot r''^{\,2}\}<c\cdot r^2$.

If $T=\{t_1,t_2,t_3\}\subseteq X_1\cup X_2$ is a triad of $M$, then let $r'$ denote the rank of $M/t_2\backslash t_3$. Note that $r'=r-1\geq 2$ holds. Our assumption, Claim~\ref{cl:triv}, Lemma~\ref{lem:cocircuit1} and Lemma~\ref{lem:cocircuit2} then imply the existence of an $F$-avoiding $\cX$-$\cY$ exchange sequence of width at most $2\cdot c\cdot (r'-1)+2\leq 2\cdot c\cdot (r-1)$ and length at most $c\cdot r'^{\,2}+2\cdot c\cdot (r'-1)+2\leq c\cdot r^2$.
\end{proof}

For $2$-sums, we get the following.

\begin{cor}\label{cor:numsteps2sum}
Let $\cX=(X_1,X_2)$ and $\cY=(Y_1,Y_2)$ be compatible pairs of disjoint bases of a matroid $M$ of rank $r\geq 3$ where $M$ contains no nontrivial tight set. Assume that for any minor $M'$ of $M$ and for any pair $\cX',\cY'$ of compatible pairs of disjoint bases of $M'$, there exists an $\cX'$-$\cY'$ exchange sequence in $M'$ of width at most $2\cdot c\cdot (r'-1)$ and length at most $c\cdot r'^{\,2}$, where $2\leq r'<r$ is the rank of $M'$ and $c\geq 1$. If $M$ is the 2-sum of two matroids, then there exists an $\cX$-$\cY$ exchange sequence in $M$ of width at most $2\cdot c\cdot (r-1)$ and length at most $c\cdot r^2$.
\end{cor}
\begin{proof}
If $M=M_\circ\oplus_2 M_\bullet$, then both $M_\circ$ and $M_\bullet$ are minors of $M$~\cite[Proposition~7.1.21]{oxley2011matroid}. Let $r'$ and $r''$ denote the ranks of $M_\circ$ and $M_\bullet$, respectively. Note that $r=r'+r''-1$ holds. 
Moreover, we claim that $r',r''\geq 2$. Indeed, e.g.\ $r' \ge 1$ and $|E_\circ|\ge 3$ hold by the definition of 2-sums, and $r'=1$ would imply that $M$ contains a nontrival tight set.
Our assumption and Lemma~\ref{lem:2sum} then imply the existence of an $\cX$-$\cY$ exchange sequence of width at most $2\cdot c\cdot(r'-1)+2\cdot c\cdot (r''-1) = 2\cdot c\cdot (r-1)$ and length at most $c\cdot r'^{\,2}+c \cdot r''^{\,2}\leq c\cdot r^2$.
\end{proof}

\section{Bounding the Number of Exchanges for Graphs}
\label{sec:graphic}

White's conjecture was settled for graphic matroids in~\cite{farber1985edge} for sequences of length two, and in~\cite{blasiak2008toric} for sequences of arbitrary length. Both results rely on the same algorithm, and in fact imply Gabow's conjecture as well for the graphic case. However, neither discusses the length of the resulting exchange sequence. 

The goal of this section is to prove strengthenings of Theorem~\ref{thm:regwhite} and Theorem~\ref{thm:reggabow} for graphs. Due to the fact that most of the work has already been done in Section~\ref{sec:reductions}, the proofs are simple and compact. Throughout the section, we use the fact that every minor of a graphic matroid is graphic again without explicitly mentioning it. 

\subsection{Quadratic Upper Bound}
\label{sec:readout}

We give the first polynomial bound on the exchange distance of compatible basis pairs in graphic matroids. In addition, through an analysis of the degree sequences of graphs that can be partitioned into two forests, we show how to exclude certain edges to participate in the exchange sequence. This observation plays a key role in the proof of Theorem~\ref{thm:regwhite}: when considering the 3-sum of a regular and a graphic matroid along a triad $T$, one can solve the two subproblems corresponding to the two sides of the 3-sum while restricting the usage of the elements of $T$, which in turn allows for an efficient merging of the sequences.

\begin{thm} \label{thm:graphicwhite} 
Let $\cX=(X_1, X_2)$ and $\cY=(Y_1, Y_2)$ be compatible pairs of bases of a graphic matroid $M$ of rank $r\geq 2$ where the underlying graph is $G=(V,E)$, and let $F \subseteq (X_1 \cap Y_1) \cup (X_2 \cap Y_2)$ be such that $|V(F)| \le 3$. Then, there exists an $F$-avoiding $\cX$-$\cY$ exchange sequence of width at most $2\cdot(r-1)$ and length at most $r^2$.
\end{thm}
\begin{proof}
We prove the theorem by induction on the rank. For $r=2$, the statement holds by Claim~\ref{cl:triv}. Therefore, we consider the case $r\geq 3$, implying that $|V|\geq 4$. Since the elements not contained in any of the bases cannot participate in an $\cX$-$\cY$ exchange sequence, we may assume without loss of generality that $X_1\cup X_2 = Y_1\cup Y_2=E$. That is, $G$ is a graph whose edge set can be partitioned into two forests. Furthermore, by the induction hypothesis and Corollary~\ref{cor:numstepstriad}, we may assume that $X_1\cap X_2=Y_1\cap Y_2=\emptyset$. 
If there are isolated vertices in the graph, then those can be deleted without changing the problem. Since $E$ can be partitioned into two forests, we have $\sum_{v\in V}d(v)=2\cdot |E|\leq 4\cdot (|V|-1)$. This implies that $G$ contains a vertex $u$ of degree at most $3$. Since both forests are bases in the graphic matroid, those are maximal forests, implying that $d(u)\geq 2$. 

Assume first that $G$ has a vertex $u$ of degree $2$. Then $u$ is a leaf vertex in both $X_1$ and $X_2$, implying $r_M(E-\delta(u))=r_M(E)-1$. Since $|E|=2\cdot r_M(E)$, we get $|E-\delta(u)|=|E|-2=2\cdot r_M(E)-2=2\cdot r_M(E-\delta(u))$. That is, $\emptyset\neq E-\delta(u)\subsetneq E$ is a tight set, and the statement follows by the induction hypothesis and Corollary~\ref{cor:numstepstriad}.

If $G$ contains no vertex of degree $2$, then it has at least four vertices of degree $3$. By condition $|V(F)|\leq 3$, there exists a vertex $u$ such that $d(u)=3$ and $\delta(u)\cap F=\emptyset$. That is, $\delta(u)$ defines a triad of $M$ disjoint from $F$, and the statement follows by the induction hypothesis and Corollary~\ref{cor:numstepstriad}.
\end{proof}

\begin{rem}
The assumption of Theorem~\ref{thm:graphicwhite} on $|V(F)|\leq 3$ is tight in the sense that an $F$-avoiding $\cX$-$\cY$ exchange sequence might not exist even if $F$ consists of a pair of disjoint edges. For an example, let $\cX=(\{a,b,c\},\{d,e,f\})$ and $\cY=(\{a,e,c\},\{d,b,f\})$ be pairs of disjoint bases of the graphic matroid of a complete graph on four vertices, see Figure~\ref{fig:k4}. Note that any symmetric exchange between $\{a,b,c\}$ and $\{d,e,f\}$ uses at least one of $b$ and $e$. Therefore, for the choice $F=\{b,e\}$, there exists no $F$-avoiding $\cX$-$\cY$ exchange sequence.
\end{rem}

\begin{figure}[t!]
\centering
\begin{subfigure}[t]{0.47\textwidth}
\centering
\includegraphics[width=.3\linewidth]{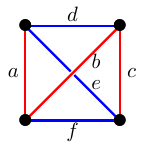}
\caption{The pair $\cX=(X_1,X_2)$ of disjoint bases of $M$.}
\label{fig:k4a}
\end{subfigure}\hfill
\begin{subfigure}[t]{0.47\textwidth}
\centering
\includegraphics[width=.3\linewidth]{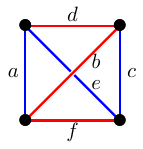}
\caption{The pair $\cY=(Y_1,Y_2)$ of disjoint bases of $M$.}
\label{fig:k4b}
\end{subfigure}
\caption{Example showing that Theorem~\ref{thm:graphicwhite} no longer holds if $F$ consists of a pair of disjoint edges. For the choice $F=\{b,e\}\subseteq (X_1\cap Y_1)\cup (X_2\cap Y_2)$, every $\cX$-$\cY$ exchange sequence uses at least one of $b$ and $e$.}
\label{fig:k4}
\end{figure}

\subsection{Strictly Monotone Sequences}
\label{sec:readout2}

Given compatible basis pairs $\cX,\cY$ of a matroid, an $\cX$-$\cY$ exchange sequence is called \emph{strictly monotone} if each step decreases the difference between the first members of the pairs. Using this terminology, Conjecture~\ref{conj:gabow} states that for any pair of disjoint bases $X_1,X_2$ of a matroid, there exists a strictly monotone exchange sequence between $(X_1,X_2)$ and $(X_2,X_1)$. 

\begin{thm} \label{thm:graphicgabow}
Let $X_1,X_2$ be disjoint bases of a graphic matroid $M$. Then, for any $h\in X_1\cup X_2$, there exists a sequence of symmetric exchanges that transforms $(X_1,X_2)$ into $(X_2,X_1)$, has length $r$ and exchanges $h$ in the last step. 
\end{thm}
\begin{proof}
Similarly to the proof of Theorem~\ref{thm:graphicgabow}, we can assume that $X_1\cup X_2=Y_1\cup Y_2=E$, and hence $G$ contains a vertex $u$ of degree at most $3$. Furthermore, if $u$ has degree $2$ then $E-\delta(u)$ is a nonempty proper tight set of $M$, while if $d(u)=3$ then $\delta(u)$ is a triad of $M$. Therefore, the statement follows by the induction hypothesis, Claim~\ref{cl:triv}, Lemma~\ref{lem:tight} and Lemma~\ref{lem:cocircuit2}.
\end{proof}

\begin{rem}
Theorem~\ref{thm:graphicgabow} settles a special case of a conjecture of Kotlar and Ziv that aims at extending the notion of serial symmetric exchanges to subsets of bases. Namely, suppose $X_1$ and $X_2$ are bases of a matroid $M$. Two subsets $A_1\subseteq X_1$ and $A_2\subset X_2$ are called \emph{serially exchangeable} if there exist orderings $A_1=\{a^1_1,\dots,a^1_q\}$ and $A_2=\{a^2_1,\dots,a^2_q\}$ such that $X_1-\{a^1_1,\dots,a^1_i\}+\{a^2_1,\dots,a^2_i\}$ and $X_2-\{a^2_1,\dots,a^2_i\}+\{a^1_1,\dots,a^1_i\}$ are bases for $i=1,\dots,q$. Kotlar and Ziv~\cite{kotlar2013serial} conjectured that for any $A\subseteq X_1$, there exists a set $B\subseteq X_2$ for which $A$ and $B$ are serially exchangeable. Note that this conjecture implies Gabow's conjecture.

As a relaxation, a matroid has the \emph{$k$-serial exchange property} for some positive integer $k$ if for any two bases $X_1,X_2$ and any subset $A_1\subseteq X_1$ of size $k$, there is a subset $A_2\subseteq X_2$ for which $A_1$ and $A_2$ are serially exchangeable. It was shown in~\cite{kotlar2013serial} that every matroid has the 2-serial exchange property. Kotlar~\cite{kotlar2013circuits} further verified that for matroids of rank at least three, for any two bases $X_1,X_2$ there exist $A_1\subseteq X_1$ and $A_2\subseteq X_2$ such that $|A_1|=|A_2|=3$ and $A_1$ and $A_2$ are serially exchangeable. Recently, McGuiness~\cite{mcguinness2022serial} showed that all binary matroids of rank at least three have the 3-serial exchange property. However, it is still unknown whether all matroids of rank at least three have the 3-exchange property.

Using this terminology, the statement of Theorem~\ref{thm:graphicgabow} is equivalent to the $(r-1)$-serial exchange property of graphic matroids.
\end{rem}

\section{Finding a Sequence of Exchanges in Polynomial Time}
\label{sec:reconf}

This section is dedicated to the proofs of Theorem~\ref{thm:regwhite} and Theorem~\ref{thm:reggabow}.  

\subsection{Preparations}
\label{sec:prep}

For proving the theorems, we need some preliminary observations. We first discuss the structure of bispanning graphs, and characterize their partitions into disjoint spanning trees in terms of intersections with a triangle. We then verify the theorems for the matroid $R_{10}$, and prove an analogous result to $F_7$ as well. Recall that the matroid $R_{10}$ is one of the basic building blocks of regular matroids, while $F_7$ is considered here to extend our results to max-flow min-cut matroids, see Section~\ref{sec:conc}. Finally, we show that it suffices to consider the problem for matroids that arise as the 3-sum of a regular matroid and the graphic matroid of a 4-regular graph.

\subsubsection{Partitions of Bispanning Graphs}
\label{sec:decomp}

Binary matroids have distinguished structural properties, which implies the following.

\begin{lem}\label{lem:0plusTchoose2} 
Let $T = \{t_1, t_2, t_3\}$ be a triangle of a binary matroid $M$ on ground set $E$. Then, the following are equivalent:
\begin{enumerate}[label=(\roman*)]\itemsep0em
\item $E - T$ partitions into a basis of $M$ and a basis of $M/T$, \label{it:decomp1}
\item $E-t_i$ partitions into two bases of $M$ for each $i \in \{1,2,3\}$,\label{it:decomp2}
\item $|E|=2\cdot r_M(E)+1$ and $|X| \le 2\cdot r_M(X)$ holds if $T \not \subseteq X\subseteq E$. \label{it:decomp3}
\end{enumerate}
\end{lem}
\begin{proof}
Condition \ref{it:decomp1} implies \ref{it:decomp2}, since if $E - T = B \cup B'$ is a partition such that $B \in \cB(M)$ and $B'\in \cB(M/T)$, then $E-t_i = B \cup (B'+T-t_i)$ is a partition into two bases of $M$ for $i \in \{1,2,3\}$. 

Condition \ref{it:decomp2} is equivalent to \ref{it:decomp3}, since $E-t_i$ partitions into two bases of $M$ if and only if $|E-t_i| = 2\cdot r_M(E)$ and $|X| \le 2\cdot r_M(X)$ holds for $X \subseteq E-t_i$. 

It remains to show that \ref{it:decomp3} implies \ref{it:decomp1}. Since $r_M(E)+r_{M/T}(E)=2\cdot r_M(E)-2 = |E|-3 = |E - T|$, it is enough to show that $E - T$ is independent in the sum of the matroids $M$ and $M/T$, which is equivalent to $r_M(X)+r_{M/T}(X) \ge |X|$ for every $X \subseteq E- T$. Since $r_{M/T}(X) = r_M(X\cup T)-r_M(T) = r_M(X\cup T)-2$, it suffices to show that \[r_M(X) + r_M(X\cup T) \ge |X|+2 \text{ for }X \subseteq E- T.\]
If $r_M(X\cup T) = r_M(X)$, then $r_M(X)+r_M(X\cup T)= 2\cdot r_M(X+t_1+t_2) \ge |X+t_1+t_2|=|X|+2$. If $r_M(X\cup T)=r_M(X)+2$, then the desired inequality follows from $2\cdot r_M(X) \ge |X|$. It remains to consider the case $r_M(X\cup T)=r_M(X)+1$. Then, $M$ being binary implies that $X$ spans at least one of $t_1$, $t_2$ and $t_3$. Indeed, if $X$ does not span $t_1$ or $t_2$, then $r_M(X+t_1+t_2)=r_M(X)+1$ implies that there is a circuit $C \subseteq X+t_1+t_2$ containing both $t_1$ and $t_2$, thus $C \triangle T \subseteq X+t_3$ is a cycle containing $t_3$, hence $X$ spans $t_3$. If $X$ spans $t_i$, then \[r_M(X)+r_M(X\cup T) = 2\cdot r_M(X)+1 = 2\cdot r_M(X+t_i)+1 \ge |X+t_i|+1 = |X|+2,\] 
concluding the proof of the lemma.
\end{proof}

\begin{rem}
We note that \ref{it:decomp2} does not necessarily imply \ref{it:decomp1} for nonbinary matroids. For example, consider the matroid on ground set $\{e_1,e_2,t_1,t_2,t_3\}$ in which $e_1$ and $e_2$ are parallel and the matroid obtained by deleting $e_1$ is the rank-2 uniform matroid. Then, $\{e_1,t_j\},\{e_2,t_k\}$ is a partition of $E-t_i$ into two bases of $M$ for any choice of indices satisfying $\{i,j,k\}=\{1,2,3\}$. However, $E-T=\{e_1,e_2\}$ consists of parallel elements in $M$ and of loops in $M/T$, hence it can not be decomposed into a basis of $M$ and a basis of $M/T$.
\end{rem}

We will use the following corollary of the lemma for graphic matroids.

\begin{cor} \label{cor:graphTchoose2} 
Let $G=(V,E)$ be a graph and $T=\{t_1,t_2,t_3\}$ be a triangle of $G$. Then, there exists a partition $E- T = F_1 \cup F_2$ such that $F_1+\{t_1,t_2\}$ and $F_2$ are disjoint spanning trees of $G$ if and only if $|E|=2\cdot |V|-1$ and 
\[|E[U]| \le \begin{cases} 2\cdot |U|-2 & \text{if } V(T) \not \subseteq U, \\
2\cdot |U|-1 & \text{if } V(T) \subseteq U
\end{cases}\]
holds for each $\emptyset\neq U \subseteq V$. 
\end{cor}
\begin{proof}
Each of the graphs $G-t_1$, $G-t_2$, $G-t_3$ decompose into two spanning trees if and only if $|E|=2\cdot |V|-1$ and any nonempty subset $U \subseteq V$ spans at most $2\cdot |U|-2$ edges in each of them. The latter condition is equivalent to $|E[U]| \le 2\cdot |U|-2$ if $V(T) \not \subseteq U$, and to $|E[U]| \le 2\cdot |U|-1$ if $V(T) \subseteq U$.  Therefore, the statement follows from the equivalence of Lemma~\ref{lem:0plusTchoose2}\ref{it:decomp1} and Lemma~\ref{lem:0plusTchoose2}\ref{it:decomp2} applied to the graphic matroid of $G$. 
\end{proof}

\begin{figure}[t!]
\centering
\begin{subfigure}[t]{0.23\textwidth}
\centering
\includegraphics[width=.75\linewidth]{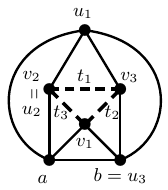}
\caption{
Graph $G=(V,E)$ and triangle $T=\{t_1,t_2,t_3\}$.
}
\label{fig:graph1}
\end{subfigure}\hfill
\begin{subfigure}[t]{0.23\textwidth}
\centering
\includegraphics[width=.75\linewidth]{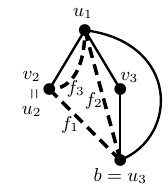}
\caption{
Graph $G'=(V',E')$ with new edges $f_1,f_2,f_3$.
}
\label{fig:graph2}
\end{subfigure}\hfill
\begin{subfigure}[t]{0.23\textwidth}
\centering
\includegraphics[width=.75\linewidth]{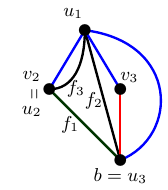}
\caption{
Partition into $F'_1$ (red) and $F'_2$ (blue).
}
\label{fig:graph3}
\end{subfigure}\hfill
\begin{subfigure}[t]{0.23\textwidth}
\centering
\includegraphics[width=.75\linewidth]{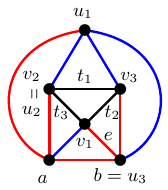}
\caption{
Construction of $F_1$ (red) and $F_2$ (blue).
}
\label{fig:graph4}
\end{subfigure}
\caption{Illustration of Lemma~\ref{lem:4reg}.}
\label{fig:graph}
\end{figure}

The proof of Theorem~\ref{thm:regwhite} will rely on the following lemma.

\begin{lem} \label{lem:4reg}
Let $G=(V,E)$ be a simple 4-regular graph and $T = \{t_1, t_2, t_3\}$ be a triangle of $G$. Assume that $|(E - T)[U]| \le 2\cdot |U|-3$ holds for each $U \subseteq V$ with $|U|\ge 2$. Then, there exists a partition $E- T = F_1 \cup F_2$ and an edge $e \in F_1$ such that each of $F_1$, $F_2+t_2$, $F_2+t_3$, $F_1 - e + t_2$, $F_1 - e + t_3$ and $F_2 + e$ is a spanning tree of $G$.
\end{lem}
\begin{proof}
Let $v_1$, $v_2$ and $v_3$ denote the vertices of $T$ such that $t_1=v_2v_3$, $t_2=v_1v_3$ and $t_3=v_1v_2$. We denote by $a$ and $b$ the neighbours of $v_1$ distinct from $v_2$ and $v_3$, and by $u_1$, $u_2$ and $u_3$ the neighbours of $a$ distinct from $v_1$. 
Let $f_i$ be a new edge between vertices $u_{i+1}u_{i+2}$, where indices are meant in a cyclic order. 
Let $G'=(V',E')$ denote the graph $G-\{v_1,a\}-t_1+\{f_1,f_2,f_3\}$.
Note that $|E'| = |E|-5 = 2\cdot |V|-5 = 2\cdot |V'|-1$. 
Consider a subset $U \subseteq V'$ with $|U| \ge 2$. If $\{u_1, u_2, u_3\} \not \subseteq U$, then $|E'[U]| \le |(E- T)[U]|+1 \le 2\cdot |U|-2$, while if $\{u_1, u_2, u_3\} \subseteq U$, then $|E'[U]| = |(E - T)[U+a]| \le 2\cdot |U+a|-3 = 2\cdot |U|-1$ holds. This shows that the conditions of 
Corollary~\ref{cor:graphTchoose2} are satisfied, hence there exists a partition $E'-\{f_1, f_2, f_3\} = F'_1 \cup F'_2$ such that  $F'_1+\{f_2,f_3\}$ and $F'_2$ are spanning trees of $G'$.  Let $F_1 \coloneqq F'_1 + \{au_1,au_2,au_3,v_1b\}$, $F_2\coloneqq F'_2+v_1a$, $e\coloneqq v_1b$. Then $F_1-e$ is a spanning tree of $G-v_1$, hence $F_1$, $F_1-e+t_2$ and $F_1-e+t_3$ are spanning trees of $G$. Since $F_2$ is a forest such that $v_1a$ is one of its two components, we get that $F_2+t_3$, $F_2+t_2$ and $F_2+e$ are also spanning trees of $G$. 
\end{proof}

\subsubsection{Solving the Problem for \texorpdfstring{$R_{10}$}{R10} and \texorpdfstring{$F_7$}{F7}}
\label{subsubsec:f7r10}

We now verify Theorem~\ref{thm:regwhite} and Theorem~\ref{thm:reggabow} for the matroid $R_{10}$. Recall that $R_{10}$ is the binary matroid represented by a matrix $A\in GF(2)^{5 \times 10}$ in which the columns are different and each of them contains exactly two zero entries.

\begin{prop} \label{prop:r10}
$R_{10}$ satisfies Theorem~\ref{thm:regwhite} and Theorem~\ref{thm:reggabow}. 
\end{prop}
\begin{proof}
Let $\cX=(X_1, X_2)$ be a basis pair of $R_{10}$. It follows from Theorem~\ref{thm:seymour} that each proper minor of $R_{10}$ is graphic or cographic, hence we may assume that $X_1$ and $X_2$ are disjoint by Lemma~\ref{lem:disjoint}.

We will use the representation of $R_{10}$ as the even-cycle matroid of the complete graph $K_5$ on vertices $\{v_1, v_2, v_3, v_4, v_5\}$, see e.g.~\cite[page 238]{oxley2011matroid}. The ground set of this matroid is the edge set of $K_5$ and the circuits are the cycles of length four and the unions of two triangles having exactly one vertex in common. To get an isomorphism between this matroid and the binary matroid represented by $A$, map an edge $v_iv_j$ to the column of $A$ in which the $i$th and $j$th entries are zero, see Figure~\ref{fig:r10graph}. The bases are the sets of five edges containing no even cycles and exactly one odd cycle of $K_5$. This implies that $(Z_1, Z_2)$ is a basis pair for a partition $E=Z_1 \cup Z_2$ if and only if $Z_1$ is a 5-cycle of $K_5$ or $Z_1 = \{v_iv_j, v_iv_k, v_jv_k, v_jv_l, v_kv_m\}$ for some $\{i,j,k,l,m\} = \{1,2,3,4,5\}$. Observe that there is an isomorphism that maps a basis of the latter form into a 5-cycle, e.g.\ $(v_1v_2, v_1v_3, v_2v_3, v_2v_4, v_3v_5)$ can be mapped to $(v_1v_3, v_1v_2, v_4v_5, v_2v_4, v_3v_5)$ by mapping $(v_1v_4, v_1v_5, v_2v_5, v_3v_4, v_4v_5)$ to $(v_1v_5, v_1v_4, v_2v_5, v_3v_4, v_2v_3)$. Therefore, to prove Theorems~\ref{thm:regwhite} and \ref{thm:reggabow} for $R_{10}$, we may assume that $X_1$ is a 5-cycle.  Figure~\ref{fig:r10} illustrates that each pair of disjoint bases of $R_{10}$ is reachable from this basis with at most 5 exchanges.
\end{proof}

\begin{figure}[t!]
\centering
\includegraphics[width=0.85\textwidth]{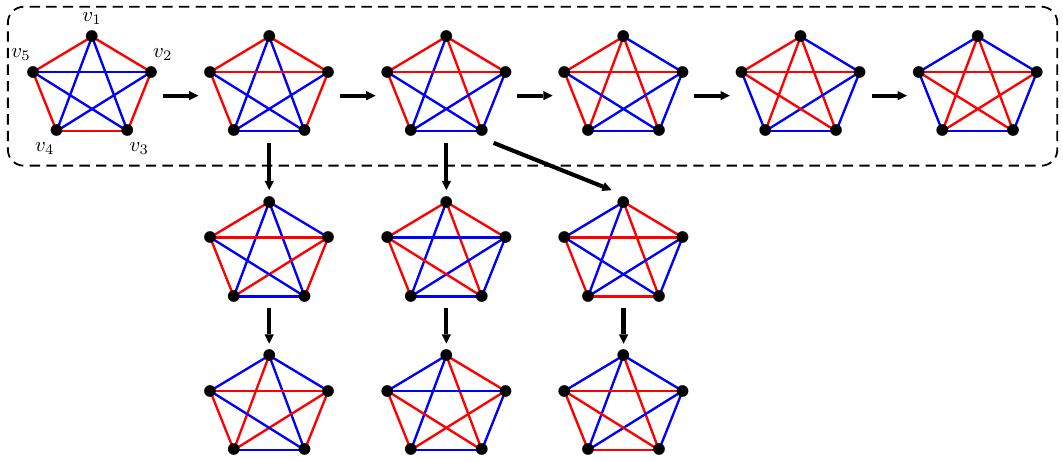}
\caption{Exchange sequences starting from a basis pair $(X_1, X_2)$ where $X_1$ and $X_2$ are 5-cycles. Basis pairs in the dashed set show a sequence of length 5 to $(X_2, X_1)$. From each basis pair of $R_{10}$, one can obtain a basis pair shown on the figure with reflections and rotations, thus each pair of bases can be obtained from $(X_1, X_2)$ with at most 5 exchanges by its symmetry.}
\label{fig:r10}
\end{figure}

Though it is not needed in the proof of Theorem~\ref{thm:regwhite} and Theorem~\ref{thm:reggabow}, we verify analogous statements for the Fano matroid as well. This will allow us to extend our results to max-flow min-cut matroids.

\begin{prop} \label{prop:f7}
For any pair $\cX,\cY$ of compatible basis pairs of the Fano matroid, there exists an $\cX$-$\cY$ exchange sequence of width at most $4$ and length at most $9$. For disjoint bases $X_1,X_2$ of the Fano matroid, there exists an $(X_1,X_2)$-$(X_2,X_1)$ exchange sequence of length $3$. Furthermore, such sequences can be determined in polynomial time.
\end{prop}
\begin{proof}
The matroid $F_7$ is an excluded minor of graphic matroids, see~\cite[Thoerem~10.3.1]{oxley2011matroid}. For any pair of bases $\cX = (X_1, X_2)$, the restriction $F_7|(X_1 \cup X_2)$ is a proper minor of $F_7$, thus it is graphic. Since the rank of $F_7$ is $3$, the statements follow from Theorems~\ref{thm:graphicwhite} and \ref{thm:graphicgabow}.    
\end{proof}

\subsubsection{Reducing the Graphic Part to 4-regular Graphs}
\label{sec:nemtommi2}

This section is devoted to proving our main structural observation. The proposition is based on a careful combination of the result of Aprile and Fiorini on decompositions trees of $3$-connected matroids (Theorem~\ref{thm:af1}), the result of McGuiness on the dual of $3$-sums (Proposition~\ref{prop:mcguinness}), our observation on the cographicness of a matroid obtained from a cographic matroid by a \dy exchange using a coindependent triangle (Lemma~\ref{lem:co}), and the reduction steps introduced in Section~\ref{sec:reductions}.

\begin{prop}\label{prop:reductions}
    Let $M$ be a $3$-connected regular matroid that is not graphic, cographic or isomorphic to $R_{10}$. Then, there exists a regular matroid $M_\circ$ and a graphic matroid $M_\bullet$ such that $M_\circ \oplus_3 M_\bullet \in \{M, M^*\}$, and such a decomposition can be determined using a polynomial number of oracle calls. Moreover, if $M$ contains no circuit or cocircuit of size at most 3, its ground set can be partitioned into two bases and it contains no nontrivial tight set, then $M_\bullet$ is the graphic matroid of a simple 4-regular graph.
\end{prop}
\begin{proof}
By Theorem~\ref{thm:af1}, $M$ can be written in the form $M=M_1 \oplus_3 M_2$, where $M_1$ is a regular matroid and $M_2$ is a graphic or cographic matroid. Moreover, if $M_2$ is not graphic, then it is the cographic matroid of a graph $G$ such that the 3-sum is taken along a triangle $T$ of the matroids which is a trivial cut of $G$. If $M_2$ is graphic, then let $M_\circ \coloneqq M_1$ and $M_\bullet \coloneqq M_2$. If $M_2$ is not graphic, then $M^*=\Delta_T(M_1) \oplus \Delta_T(M_2)$ by Proposition~\ref{prop:mcguinness}, where $\Delta_T(M_2)$ is a cographic matroid by Lemma~\ref{lem:co}, thus $M_\circ \coloneqq \Delta_T(M_1)^*$ and $M_\bullet \coloneqq \Delta(M_2)^*$ satisfy the requirements of the first part of the statement. 

Assume now that $M$ satisfies all the conditions of the proposition. Let $E_\circ$ and $E_\bullet$ denote the ground sets of $M_\circ$ and $M_\bullet$, respectively, and let $T\coloneqq E_\bullet \cap E_\circ$. The ground set of $M$ can be partitioned into two bases and $M$ contains no nontrivial tight set, hence the same holds for $M^*$ as well. Since $M_\circ \oplus_3 M_\bullet \in \{M, M^*\}$ contains no nontrivial tight set, the restrictions of $M_\circ \oplus_3 M_\bullet$ and $M_\bullet$ to $E_\bullet - T$ are the same, and $T$ is coindependent in $M_\bullet$, we get \[|E_\bullet - T| \le 2\cdot r_{M_\circ \oplus_3 M_\bullet}(E_\bullet - T)-1 = 2\cdot r_{M_\bullet}(E_\bullet - T) -1 = 2\cdot r_{M_\bullet}(E_\bullet)-1.\]

Let $G=(V,E_\bullet)$ be a connected graph whose graphic matroid is $M_\bullet$ and let $v_1$, $v_2$ and $v_3$ denote the vertices of the triangle $T$. We define $V' \coloneqq V- \{v_1, v_2, v_3\}$. It follows from the description of $\cB(M_\circ \oplus_3 M_\bullet)$ that the cocircuits of $M_\circ \oplus_3 M_\bullet$ and $M_\bullet$ contained in $E_\bullet - T$ are the same. Indeed, a subset of $E_\bullet- T$ intersects each basis of $M_\bullet$ if and only if it intersects each basis of $M_\circ \oplus_3 M_\bullet$. Since each cocircuit of $M_\circ \oplus_3 M_\bullet$ has size at least 4,  each vertex in $V'$ has degree at least 4 in $G$. As $M_\bullet - T$ contains no nontrivial tight set, $E_\bullet - T$ contains no parallel edges. These imply $|V'|\ge 2$, since $V'=\emptyset$ would contradict $|E_\bullet - T| \ge 4$, and $|V'|=1$ would contradict that the vertices in $V'$ have degree at least 4. Since $|V'|\ge 2$ and $M_\bullet \backslash T$ contains no nontrivial tight set, $|E_\bullet[V']| \le 2\cdot |V'|-3$ follows. Therefore,
\begin{align} \label{eq:degcount} \tag{$*$}
4\cdot |V'| \le \sum_{v \in V'} d(v) =  2\cdot|E_\bullet[V']| + \sum_{i=1}^3 d(v_i) -2\cdot |E_\bullet[\{v_1, v_2, v_3\}]| \le 2 \cdot (2\cdot |V'|-3) + \sum_{i=1}^3 d(v_i) - 6,
\end{align}
thus $\sum_{i=1}^3 d(v_i) \ge 12$. Using that $r_{M_\bullet}(E_\bullet)=|V|-1=|V'|+2$, we obtain \[2\cdot |E_\bullet| = \sum_{v \in V'} d(v) + \sum_{i=1}^3 d(v_i) \ge 4 |V'|+ 12 = 4\cdot r_{M_\bullet}(E_\bullet) + 4, \]  which yields $|E_\bullet - T|  = |E_\bullet| -3 \ge 2 r_{M_\bullet}(E_\bullet)-1$. As we have already shown that $|E_\bullet - T| \le 2 r_{M_\bullet}(E_\bullet)-1$, we obtain $|E_\bullet - T| = 2\cdot r_{M_\bullet}(E_\bullet)-1$ and all the inequalities in \eqref{eq:degcount} hold with equality. This implies that $d(v)=4$ for each $v \in V'$, $|E_\bullet[V']| = 2\cdot |V'|-3$, $\sum_{i=1}^3 d(v_i) = 12$ and $|E_\bullet[\{v_1, v_2, v_3\}]| = 3$. The last equality implies that the only edges spanned by $\{v_1, v_2, v_3\}$ are the edges of $T$, hence $G$ is a simple graph, as $E_\bullet - T$ contains no parallel edges.

It remains to show that $d(v_1)=d(v_2)=d(v_3)=4$.  Since $|E_\bullet[V'+v_i]| \le 2\cdot |V'+v_i| -3 = 2\cdot |V'|-1$ holds for $i \in \{1,2,3\}$, we get \[|E_\bullet[V']| + 2 = 2\cdot |V'|-1  \ge |E_\bullet[V'+v_i]| = |E_\bullet[V']| + d(v_i)-2,\] where the last equality follows from $E_\bullet[\{v_1, v_2, v_3\}] = T$. This yields $d(v_i)\le 4$ for $i \in \{1,2,3\}$, hence $\sum_{i=1}^3 d(v_i) = 12$ implies that $d(v_1)=d(v_2)=d(v_3)=4$.
\end{proof}

\begin{rem}\label{rem:construction}
Though Proposition~\ref{prop:reductions} significantly reduces the number of matroids for which Theorem~\ref{thm:regwhite} and Theorem~\ref{thm:reggabow} need to be verified, there indeed exist regular matroids for which none of the reduction steps apply. As an example, let $G$ be a complete bipartite graph with color classes $A=\{a_1,a_2,a_3\}$ and $B=\{b_1,b_2,b_3,b_4\}$. Furthermore, let $H$ be the graph obtained from $G$ by deleting the edge $a_1b_1$ and adding three new vertices $w_1,w_2,w_3$ with the extra edges $w_1a_1,w_1b_1,w_2b_1,w_2b_2,w_3b_3,w_3b_4,w_1w_2,w_1w_3$ and $w_2w_3$. Let $H_i$ be a copy of $H$ for $1 \le i \le 4$ with vertices $a^i_1, a^i_2, a^i_3, b^i_1, b^i_2, b^i_3, b^i_4, w^i_1, w^i_2, w^i_3$. 
Let $M_0$ denote the cographic matroid of $G$, and $N_i$ the graphic matroid of $H_i$ for $1\leq i\leq 4$. Note that $\delta(b_i)$ is a coindependent triangle of $M$ and $\{w^i_1w^i_2,w^i_1w^i_3,w^i_2w^i_3\}$ is a coindependent triangle of $N_i$ for $1\leq i\leq 4$.

Let $M$ denote the matroid that is obtained by taking the $3$-sum of $M_0$ with the $N_i$s in an arbitrary order; see Figure~\ref{fig:construction} for an illustration. Here the $3$-sum with $N_i$ is along $\delta(b_i)$ and $\{w^i_1w^i_2,w^i_1w^i_3,w^i_2w^i_3\}$, where the edge $b_ia_j$ is identified with $w^i_jw^i_{j+1}$ for $j=1,2,3$ (indices are in a cyclic order). It is not difficult to check that the ground set of $M$ partitions into two disjoint bases. Also, an easy case analysis shows that none of the reduction operations can be applied to $M$ and $M^*$, i.e., both $M$ and $M^*$ are 3-connected and contain neither a tight set nor a cocircuit of size at most three. Observe that the example is consistent with the statement of Proposition~\ref{prop:reductions} in that $N_i$ is a graphic matroid of a 4-regular graph for $1\leq i\leq 4$.

Interestingly, the bound on the minimum size of a cocircuit of $M$ and $M^*$ is tight. More precisely, if $M$ is a regular matroid whose ground set partitions into two disjoint bases, then it contains a circuit or cocircuit of size at most $4$. This immediately follows from Proposition~\ref{prop:reductions} if $M$ is 3-connected and contains no nontrivial tight set; the remaining cases can be verified using Seymour's decomposition.
\end{rem}

\begin{figure}[t!]
    \centering
    \includegraphics[width=0.7\textwidth]{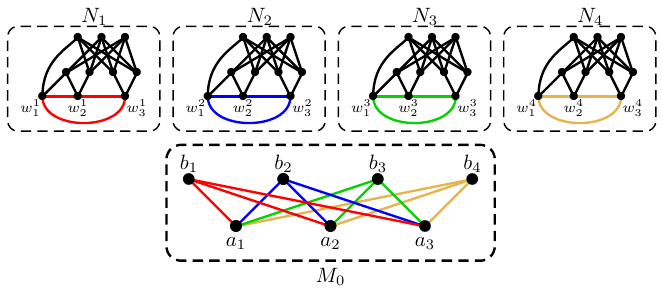}
    \caption{Illustration of Remark~\ref{rem:construction}, where $M_0$ is the cographic matroid of $G$ and $N_i$ is the graphic matroid of $H_i$ for $1\leq i\leq 4$. The matroid $M$ is obtained by taking the 3-sum of $M_0$ with the $N_i$s in an arbitrary order along the triangles having the same color.}
    \label{fig:construction}
\end{figure}

\subsection{Proof of White's Conjecture for Basis Pairs of Regular Matroids}
\label{sec:main}

\begin{proof}[Proof of Theorem~\ref{thm:regwhite}]    
We prove by induction on the rank of the matroid. The statement holds when $r_M(E) \le 2$ by Claim~\ref{cl:triv}, therefore we consider the case $r_M(E) \ge 3$. We may assume that $X_1 \cup X_2 = Y_1 \cup Y_2 = E$ by deleting the elements in $E- (X_1 \cup X_2)$. Using the induction hypothesis and Corollary~\ref{cor:numstepstriad} with $c=2$ and $F=\emptyset$, we may assume that $X_1 \cap X_2 = Y_1 \cap Y_2 = \emptyset$, $M$ contains no nontrivial tight set and contains no triad. As $E$ can be partitioned into two bases, the pairs of disjoint bases of $M$ and $M^*$ are the same, hence we may also assume that $M^*$ contains no triad, that is, $M$ contains no triangle. Note that the assumption that $M$ contains no nontrivial tight set implies that it is 2-connected and has no circuit of cocircuit of size at most two. Using the induction hypothesis and Corollary~\ref{cor:numsteps2sum} with $c=2$, we may assume that $M$ is 3-connected. Then, by Proposition~\ref{prop:reductions}, there exist regular matroids $M_\circ$ and $M_\bullet$ such that $M_\circ \oplus_3 M_\bullet \in \{M, M^*\}$ and $M_\bullet$ is the graphic matroid of a simple 4-regular graph $G$. Since the pairs of disjoint bases of $M$ and $M^*$ are the same, we may assume that $M_\circ \oplus_3 M_\bullet = M$. 

Let $E_\circ$ and $E_\bullet$ denote the ground sets of $M_\circ$ and $M_\bullet$, respectively. Define $T \coloneqq E_\circ \cap E_\bullet$ and let $V$ denote the set of vertices of $G$. Since $M$ contains no nontrivial tight set, $|(E- T)[U]| \le 2\cdot |U|-3$ holds for each $U \subseteq V$ with $|U|\ge 2$. As $G$ is 4-regular, $|E_\bullet- T| = |E_\bullet|-3 = 2\cdot |V|-3 = 2\cdot r_{M_\bullet}(E_\bullet)-1$, hence $|E_\circ - T| = 2\cdot r_{M_\circ}(E_\circ)-3$. The characterization of $\cB(M_\circ \oplus_3 M_\bullet)$ implies that for a partition $Z_1 \cup Z_2$ of $E$,  $\cZ=(Z_1, Z_2)$ is a pair of disjoint bases of $M$ if and only if
\begin{typelist}\itemsep0em
\item $Z^\circ_1+\{t_1,t_2\}$, $Z^\circ_2+t_i$ and $Z^\circ_2+t_j$ are bases of $M_\circ$, $Z^\bullet_1$, $Z^\bullet_2+t_i$ and $Z^\bullet_2+t_k$ are bases of $M_\bullet$ for some $i,j,k$ with $\{i,j,k\}=\{1,2,3\}$, or \label{it:type1}
\item $Z^\circ_1+t_i$, $Z^\circ_1+t_j$, and $Z^\circ_2+\{t_1,t_2\}$ are bases of $M_\circ$, $Z^\bullet_1+t_i$, $Z^\bullet_1+t_k$ and $Z^\bullet_2$ are bases of $M_\bullet$ for some $i,j,k$ with $\{i,j,k\}=\{1,2,3\}$. \label{it:type2}
\end{typelist}
Note that $|Z^\circ_1| = r_{M_\circ}(E_\circ)-2$ if the pair $\cZ$ is of Type~\ref{it:type1}, and $|Z^\circ_1| = r_{M_\circ}(E_\circ)-1$ if $\cZ$ is of Type~\ref{it:type2}.

For a pair $\cZ=(Z_1, Z_2)$ of disjoint basis of $M$, let $\{i_\cZ, j_\cZ, k_\cZ\} = \{1,2,3\}$ be such that  $Z^\circ_2 + t_{i_\cZ}, Z^\circ_2+t_{j_\cZ} \in \cB(M_\circ)$ and $Z^\bullet_2+t_{i_\cZ}, Z^\bullet_2+t_{k_\cZ} \in \cB(M_\bullet)$ if $\cZ$ is of Type~\ref{it:type1}, and $Z^\circ_1+t_{i_\cZ}, Z^\circ_1+t_{j_\cZ} \in \cB(M_\circ)$ and $Z^\bullet_1+t_{i_Z}, Z^\bullet_1+t_{k_Z} \in \cB(M_\bullet)$ if $\cZ$ is of Type~\ref{it:type2}. By Lemma~\ref{lem:0or2}, the indices $i_\cZ$, $j_\cZ$ and $k_\cZ$ are uniquely defined. By symmetry, we may assume that $|X^\circ_1| = r_{M_\circ}(E_\circ)-2$ is of Type~\ref{it:type1}. We may also assume that $1 \in \{i_\cX, j_\cX\} \cap \{i_\cY, j_\cY\}$. By Lemma~\ref{lem:4reg}, there exists a partition $E^\bullet - T = F^\bullet_1 \cup F^\bullet_2$ and an edge $e \in F^\bullet_1$ such that each of $F^\bullet_1$, $F^\bullet_2+t_2$, $F^\bullet_2+t_3$, $F^\bullet_1-e+t_2$, $F^\bullet_1-e+t_3$ and $F^\bullet_2+e$ is a basis of $M_\bullet$. Moreover, $X^\circ_1+t_1+t_2$ and $X^\circ_2+t_1$ are bases of $M_\circ$, hence $(X^\circ_1 \cup F^\bullet_1, X^\circ_2 \cup F^\bullet_2)$ is a pair of disjoint bases of $M$. Similarly, by letting $(\widetilde{F}^\bullet_1, \widetilde{F}^\bullet_2) \coloneqq (F^\bullet_1, F^\bullet_2)$ if $\cY$ is of Type~\ref{it:type1} and $(\widetilde{F}^\bullet_1, \widetilde{F}^\bullet_2) \coloneqq (F^\bullet_1-e, F^\bullet_2+e)$ if $\cY$ is of Type~\ref{it:type2},  $(Y^\circ_1 \cup \widetilde{F}^\bullet_1, Y^\circ_2 \cup \widetilde{F}^\bullet_2)$ is a pair of disjoint bases of $M$. We construct an $\cX$-$\cY$ exchange sequence by concatenating exchange sequences

\begin{enumerate}[label=(\arabic*)]\itemsep0em
\item \label{seq:1} from $\cX$ to $(X^\circ_1 \cup F^\bullet_1, X^\circ_2 \cup F^\bullet_2)$,
\item \label{seq:2} from $(X^\circ_1 \cup F^\bullet_1, X^\circ_2 \cup F^\bullet_2)$ to $(Y^\circ_1 \cup \widetilde{F}^\bullet_1, Y^\circ_2 \cup \widetilde{F}^\bullet_2)$, and 
\item \label{seq:3} from $(Y^\circ_1 \cup \widetilde{F}^\bullet_1, Y^\circ_2 \cup \widetilde{F}^\bullet_2)$ to $\cY$.
\end{enumerate}

For sequence~\ref{seq:1}, consider the pairs $\cX' \coloneqq (X^\bullet_1, X^\bullet_2+t_{k_\cX})$ and $\cF'\coloneqq (F^\bullet_1, F^\bullet_2+t_{k_\cX})$ of disjoint bases of the graphic matroid $M_\bullet$. By Theorem~\ref{thm:graphicwhite}, there exists a $\{t_{k_\cX}\}$-avoiding $\cX'$-$\cF'$ exchange sequence in $M_\bullet$ of width at most $2\cdot (r_{M_\bullet}(E_\bullet)-1)$ and length at most $r_{M_\bullet}(E_\bullet)^2$. As none of the exchanges use the element $t_{k_\cX}$, each basis pair $\cZ'$ of the sequence can be written as $\cZ'=(Z^\bullet_1, Z^\bullet_2 + t_{k_\cX})$. Consider the same symmetric exchanges applied to $\cX$ in $M$, that is, for a basis pair $(Z^\bullet_1, Z^\bullet_2 + t_{k_\cX})$ of the $\cX'$-$\cF'$ sequence, consider the basis pair $(X^\circ_1 \cup Z^\bullet_1, X^\circ_2 \cup Z^\bullet_2)$ of $M$. This gives an exchange sequence from $\cX$ to $(X^\circ_1 \cup F^\bullet_1, X^\circ_2 \cup F^\bullet_2)$ of width at most $2\cdot (r_{M_\bullet}(E_\bullet)-1)$ and length at most $r_{M_\bullet}(E_\bullet)^2$.

For sequence~\ref{seq:2}, let $\cX'' \coloneqq (X^\circ_1+t_2, X^\circ_2)$ and $\cY'' \coloneqq (Y^\circ_1 + t_2, Y^\circ_2)$ if $\cY$ is of Type~\ref{it:type1}, and $\cY'' \coloneqq (Y^\circ_1, Y^\circ_2 + t_2)$ of $\cY$ is of Type~\ref{it:type2}. Both $\cX''$ and $\cY''$ are pairs of disjoint bases of the regular matroid $M_\circ/t_1$, hence, by the induction hypothesis, there exists an $\cX''$-$\cY''$ exchange sequence in $M_\circ/t_1$ of width at most $4\cdot r_{M_\circ / t_1}(E_\circ-t_1)$ and length at most $2\cdot r_{M_\circ / t_1}(E_\circ-t_1)^2$. We perform the steps of this $\cX''$-$\cY''$ exchange sequence on $(X^\circ_1 \cup F^\bullet_1, X^\circ_2 \cup F^\bullet_2)$, but whenever a symmetric exchanges uses $t_2$ and some other element $f$, then exchange $e$ and $f$ instead. Formally, if a symmetric exchange transforms $(Z^\circ_1+t_2, Z^\circ_2)$ into $(Z^\circ_1+f, Z^\circ_2-f+t_2)$, then this step is replaced by transforming $(Z^\circ_1 \cup F^\bullet_1, Z^\circ_2 \cup F^\bullet_2)$ into $((Z^\circ_1+f)\cup (F^\bullet_1-e), (Z^\circ_2-f)\cup (F^\bullet_2+e))$. Note that $(Z^\circ_1 \cup F^\bullet_1, Z^\circ_2 \cup F^\bullet_2)$ is a basis of $M$. Indeed, this follows from $Z^\circ_1+\{t_1,t_2\} \in \cB(M_\circ)$, $F^\bullet_1 \in \cB(M_\bullet)$ and $Z^\circ_2+t_1 \in \cB(M_\circ)$, $F^\bullet_2 + t_2, F^\bullet_2 + t_3 \in \cB(M_\bullet)$. Analogously, $((Z^\circ_1+f)\cup (F^\bullet_1-e), (Z^\circ_2-f)\cup (F^\bullet_2+e))$ is a basis pair of $M$. This follows from $Z^\circ_1+\{f,t_1\} \in \cB(M_\circ)$, $F^\bullet_1-e+t_2, F^\bullet_1-e+t_3 \in \cB(M_\bullet)$ and $Z^\circ_2-f+\{t_1,t_2\} \in \cB(M_\circ)$, $F^\bullet_2 + e \in \cB(M_\bullet)$. Similarly, if a symmetric exchange transforms $(Z^\circ_1, Z^\circ_2+t_2)$ into $(Z^\circ_1-f+t_2, Z^\circ_2+f)$, then this step is replaced by transforming $(Z^\circ_1 \cup (F^\bullet_1-e), Z^\circ_2 \cup (F^\bullet_2+e))$ into $((Z^\circ_1-f)\cup F^\bullet_1, (Z^\circ_2+f)\cup F^\bullet_2)$. This gives an exchange sequence from $(X^\circ_1 \cup F^\bullet_1, X^\circ_2 \cup F^\bullet_2)$ to $(Y^\circ_1 \cup \widetilde{F}^\bullet_1, Y^\circ_2 \cup \widetilde{F}^\bullet_2)$ in $M$ of width at most $4\cdot (r_{M_\circ / t_1}(E_\circ-t_1)-1)$ and length at most $2\cdot r_{M_\circ / t_1}(E_\circ-t_1)^2$.

For sequence~\ref{seq:3}, the construction is analogous to that of sequence~\ref{seq:1}. Let $\cF''' \coloneqq (\widetilde{F}^\bullet_1, \widetilde{F}^\bullet_2 + t_{k_\cY})$ and $\cY''' \coloneqq (Y^\bullet_1, Y^\bullet_2+t_{k_\cY})$ if $\cY$ is of Type~\ref{it:type1}, and let $\cF''' \coloneqq (\widetilde{F}^\bullet_1 + t_{k_\cY}, \widetilde{F}^\bullet_2)$ and $\cY''' \coloneqq (Y^\bullet_1 +t_{k_\cY}, Y^\bullet_2)$ if $\cY$ is of Type~\ref{it:type2}. By Theorem~\ref{thm:graphicwhite}, there exists a $\{t_{k_\cY}\}$-avoiding $\cF'''$-$\cY'''$ exchange sequence in $M_\bullet$ of width at most $2\cdot (r_{M_\bullet}(E_\bullet)-1)$ and length at most $r_{M_\bullet}(E_\bullet)^2$, which yields an exchange sequence of the same width an length in $M$ from $(Y^\circ_1 \cup \widetilde{F}^\bullet_1, Y^\circ_2 \cup \widetilde{F}^\bullet_2)$ to $\cY$.

Concatenating the three sequences we obtain an $\cX$-$\cY$ exchange sequence in $M$ of width at most 
\begin{align*}
{}&{} 2\cdot (r_{M_\bullet}(E_\bullet)-1)+4\cdot (r_{M_\circ/t_1}(E_\circ-t_1)-1)+2\cdot (r_{M_\bullet}(E_\bullet)-1)\\
{}&{}= 4\cdot (r_{M_\circ}(E_\circ) + r_{M_\bullet}(E_\bullet)-3)\\
{}&{}= 4\cdot (r_M(E)-1),
\end{align*}
and length at most 
\begin{align*}
{}&{}r_{M_\bullet}(E_\bullet)^2+2\cdot r_{M_\circ/t_1}(E_\circ-t_1)^2+r_{M_\bullet}(E_\bullet)^2\\
{}&{}= 2 \cdot (r_{M_\bullet}(E_\bullet)^2 + (r_{M_\circ}(E_\circ)-1)^2) \\
{}&{}\leq 2 \cdot (r_{M_\bullet}(E_\bullet)+r_{M_\circ}(E_\circ)-2)^2\\
{}&{}= 2\cdot r_M(E)^2,
\end{align*}
where $r_{M_\bullet}(E_\bullet)^2 + (r_{M_\circ}(E_\circ)-1)^2 \le (r_{M_\bullet}(E_\bullet)+r_{M_\circ}(E_\circ)-2)^2$ holds by $r_{M_\bullet}(E_\bullet), r_{M_\circ}(E_\circ) \ge 3$. 

The proof leads to a polynomial algorithm for determining a $\cX$-$\cY$ exchange sequence. We have already discussed that one can find a decomposition of $M$ into basic matroids and perform the reduction steps using a polynomial number of oracle calls. Once a decomposition of the form $M=M_\circ\oplus_3 M_\bullet$ is identified where $M_\bullet$ is a graphic matroid of $G=(V,E)$, the proof of Lemma~\ref{lem:4reg} shows how to find a partition $E-T=F_1\cup F_2$. Finally, above we described how to construct the $\cX$-$\cY$ exchange sequence based on these, concluding the proof of the theorem.
\end{proof}

\subsection{Proof of Gabow's Conjecture for Regular Matroids}
\label{sec:main2}

\begin{proof}[Proof of Theorem~\ref{thm:reggabow}]
We prove by induction on the rank of the matroid. The statement holds when $r_M(E)\le 2$ by Claim~\ref{cl:triv}, therefore we consider the case $r_M(E) \ge 3$. We may assume that $X_1 \cup X_2 = E$ by deleting the elements in $E- (X_1 \cup X_2)$. 
If $M$ contains a nontrivial tight set $Z$, then we apply the induction hypothesis to $M|Z$ and $M/Z$ and use Lemma~\ref{lem:disjoint} together with the fact that $r(M|Z) + r(M/Z) = r_M(E)$.  If $M=M_\circ \oplus_2 M_\bullet$, then we apply the induction hypothesis to $M_\circ$ and $M_\bullet$ and use Lemma~\ref{lem:2sum} together with the fact that $r_{M_\circ}(E_\circ)+r_{M_\bullet}(E_\bullet)-1 = r_M(E)$. If $M$ contains a triad $T=\{t_1, t_2, t_3\}$ where, say, $t_1, t_2 \in X_1$, then there exists an exchange sequence from $(X_1-t_2, X_2-t_3)$ to $(X_2-t_2, X_1-t_3)$ in $M/t_2 \backslash t_3$ of length $r_{M/t_2 \backslash t_3}(E-\{t_2,t_3\}) = r_M(E)-1$ by the induction hypothesis. Hence, the statement follows using Lemma~\ref{lem:cocircuit2}. Similarly, we are also done if $M$ contains a triangle by taking the dual of $M$. Therefore, we may assume that all conditions of Proposition~\ref{prop:reductions} hold and, by taking the dual of $M$ if necessary, $M=M_\circ \oplus_3 M_\bullet$ where $M_\circ$ is a regular matroid and $M_\bullet$ is the graphic matroid of a simple 4-regular graph $G$. 

Similarly to the proof of Theorem~\ref{thm:regwhite}, the pairs of disjoint bases of $M$ can be of Type~\ref{it:type1} and Type~\ref{it:type2}. We may assume by symmetry that $(X_1, X_2)$ is of Type~\ref{it:type1} with $i=1$, $j=2$ and $k=3$, that is, $X^\circ_1+t_1+t_2$, $X^\circ_2+t_1$ and $X^\circ_2+t_2$ are bases of $M_\circ$, and $X^\bullet_1$, $X^\bullet_2+t_1$ and $X^\bullet_2+t_3$ are bases of $M_\bullet$. By the induction hypothesis, there exists an exchange sequence of length $r_{M_\circ/t_2}(E_\circ-t_2)$ transforming $(X^\circ_1+t_1, X^\circ_2)$ into $(X^\circ_2, X^\circ_1+t_1)$ in $M_\circ/t_2$. Exactly one of these exchanges uses $t_1$, say the $(\ell+1)$th step transforms $(Y^\circ_1+t_1, Y^\circ_2)$ into $(Y^\circ_1+e, Y^\circ_2-e+t_1)$  We construct an exchange sequence from $(X_1, X_2)$ to $(X_2, X_1)$ in $M$ by concatenating exchange sequences

\begin{enumerate}[label=(\arabic*)]\itemsep0em
\item \label{seqq:1} from $(X_1, X_2)$ to $(Y_1^\circ\cup X^\bullet_1, Y^\circ_2 \cup X^\bullet_2)$,
\item \label{seqq:2} from $(Y_1^\circ\cup X^\bullet_1, Y^\circ_2 \cup X^\bullet_2)$ to $((Y_1^\circ+e)\cup X^\bullet_2, (Y^\circ_2-e) \cup X^\bullet_1)$, and
\item \label{seqq:3} from $((Y_1^\circ+e)\cup X^\bullet_2, (Y^\circ_2-e) \cup X^\bullet_1)$, to $(X_2, X_1)$.
\end{enumerate}

For sequence~\ref{seqq:1}, we apply the $\ell$ symmetric exchanges transforming $(X^\circ_1+t_1, X^\circ_2)$ into $(Y^\circ_1+t_1, Y^\circ_2)$ in $M/t_2$ to the basis pair $(X_1, X_2)$ in $M$. As none of these exchanges uses $t_1$, each member of this sequence in $M/t_2$ can be written as $(Z^\circ_1+t_1, Z^\circ_2)$, and thus $(Z^\circ_1\cup X^\bullet_1, Z^\circ_2\cup X^\bullet_2)$ is a basis pair of $M$, since $X^\bullet_1$, $X^\bullet_2+t_1$ and $X^\bullet_2+t_3$ are bases of $M_\bullet$.

For sequence~\ref{seqq:2}, by $Y^\circ_2+t_2 \in \cB(M_\circ)$, Lemma~\ref{lem:0or2} implies that there exist $i,j$ satisfying $\{i,j\}=\{1,3\}$, $Y^\circ_2+t_i \in \cB(M_\circ)$ and $Y^\circ_2+t_j \not \in \cB(M_\circ)$. By Theorem~\ref{thm:graphicgabow}, there exists an exchange sequence of length $r_{M_\bullet}(E_\bullet)$ transforming $(X^\bullet_1, X^\bullet_2+t_j)$ into $(X^\bullet_2+t_j, X^\bullet_1)$ in $M_\bullet$ using $t_j$ in the last step. Assume that this last step exchanges $f$ and $t_j$, that is, it transforms $(X^\bullet_2+f, X^\bullet_1+t_j-f)$ into $(X^\bullet_2+t_j, X^\bullet_1)$. We apply the steps of this sequence in $M_\bullet$ except the one exchanging $f$ and $t_2$ to $(Y^\circ_1 \cup X^\bullet_1, Y^\circ_2 \cup X^\bullet_2)$. As none of these exchanges use $t_j$, each member of this sequence in $M_\bullet$ can be written as $(Z^\bullet_1, Z^\bullet_2+t_j)$, and thus $(Y^\circ_1 \cup Z^\bullet_1, Y^\circ_2 \cup Z^\bullet_2)$ is a basis pair of $M$. This way we transform $(Y^\circ_1 \cup X^\bullet_1, Y^\circ_2 \cup X^\bullet_2)$ into $(Y^\circ_1 \cup (X^\bullet_2+f), Y^\circ_2 \cup (X^\bullet_1-f))$ in $M$ using $r_{M_\bullet}(E_\bullet)-1$ steps. Finally, we obtain $((Y_1^\circ+e)\cup X^\bullet_2, (Y^\circ_2-e) \cup X^\bullet_1)$ by exchanging $e$ and $f$.

For sequence~\ref{seqq:3}, the construction is analogous to that of sequence~\ref{seqq:1}. We apply the $r(M_\circ/t_2)-\ell-1$ symmetric exchanges transforming $(Y_1^\circ+e, Y^\circ_2-e+t_1)$ into $(X^\circ_2, X^\circ_1+t_1)$ in $M/t_2$ to the basis pair $((Y_1^\circ+e)\cup X^\bullet_2, (Y^\circ_2-e) \cup X^\bullet_1)$ in $M$. The concatenation of the three sequences transforms $(X_1, X_2)$ into $(X_2, X_1)$ using  $\ell+r_{M_\bullet}(E_\bullet)+(r(M_\circ/t_2)-\ell-1) = r_{M_\bullet}(E_\bullet)+r_{M_\circ}(E_\circ)-2 = r_M(E)$ symmetric exchanges.

Similarly to the proof of Theorem~\ref{thm:regwhite}, the proof leads to a polynomial algorithm for determining an $(X_1,X_2)$-$(X_2,X_1)$ exchange sequence, concluding the proof of the theorem.
\end{proof}

\section{Conclusions}
\label{sec:conc}

In this work, we verified two long-standing open conjectures on the exchange distance of basis pairs in regular matroids. We presented preprocessing steps that reduce the problems to 3-connected regular matroids not containing small cocircuits. This led to a polynomial upper bound on the number of exchanges needed to transform a basis pair of a graphic matroid into another, a result that is of independent combinatorial interest. By combining a recent refinement of Seymour's decomposition theorem given by Aprile and Fiorini, a result of McGuinness on the dual of the 3-sum of two matroids, and 
an observation on \dy exchanges in cographic matroids,
we showed that the regular matroid can be assumed to have the form $M=M_\circ\oplus_3 M_\bullet$ where $M_\bullet$ is a graphic matroid. This, together with the aforementioned observations for the graphic case, allowed us to bound the exchange distance of basis pairs of regular matroids in general. Our proof implies an algorithm for determining a sequence of symmetric exchanges that transforms a given pair of bases into another using a polynomial number of oracle calls.

Our proof technique allows us to go beyond regular matroids to max-flow min-cut (MFMC) matroids, introduced by Seymour~\cite{seymour1977matroids} as matroids satisfying a generalization of Menger's theorem on the edge-connectivity of undirected graphs. Seymour~\cite{seymour1977matroids, seymour1981matroids} (see also \cite[Corollary~12.3.22]{oxley2011matroid}) showed that any MFMC matroid can be constructed by taking 1- and 2-sums of regular matroids and the Fano matroid $F_7$. 
Therefore, in order to extend our results to the class of MFMC matroids, it suffices to verify counterparts of Theorem~\ref{thm:regwhite} and Theorem~\ref{thm:reggabow} for $F_7$. This follows by Proposition~\ref{prop:f7}, hence our results hold for MFMC matroids as well. 

We close the paper by mentioning some open problems:

\begin{enumerate}\itemsep0em
    \item Our main motivation for considering regular matroids was Seymour's decomposition theorem. However, out technique might be applicable to any class of matroids whose members have a decomposition into basic matroids using 1-, 2- and 3-sums where the exchange distance of basis pairs can be bounded in every basic matroid. In particular, this allowed us to extend our results to MFMC matroids. It would be interesting to identify further matroid classes that admit such decompositions. 
    \item Seymour's definition of matroid sums breaks down for nonbinary matroids. A vast amount of work has focused on extending this notion to nonbinary matroids as well, see e.g.~\cite{truemper1992matroid}. Hence a natural question is whether our approach can be applied to these more general definitions of sums. 
    \item Blasiak~\cite{blasiak2008toric} settled White's conjecture for sequences of arbitrary length in graphic matroids.
    \begin{enumerate}\itemsep0em
        \item The proof recursively reduces the size of the problem by decreasing the number of bases in $(X_1,\dots,X_k)$ by one, and so it does not lead to an efficient algorithm for determining a sequence of symmetric exchanges between two basis sequences. This raises the following: does there exists a polynomial bound on the exchange distance of basis sequences of graphic matroids in general?
        \item While we could verify White's conjecture for basis pairs in regular matroids, the problem remains open for longer sequences. We believe that such a result might follow by combining our techniques with Blasiak's approach for the graphic case.
    \end{enumerate}
    \item A common generalization of White's conjecture for sequences of length two and Gabow's conjecture was proposed by Hamidoune~\cite{cordovil1993bases}, suggesting that the exchange distance of compatible basis pairs is at most the rank of the matroid. In~\cite{berczi2022weighted}, the conjecture was verified for strongly base orderable matroids, split matroids, spikes, and graphic matroids of wheel graphs. However, it remains open even for graphic matroids in general. 
\end{enumerate}

 \paragraph{Acknowledgement} 
 The authors are grateful to Dániel Garamvölgyi and Yutaro Yamaguchi for initial discussions on the problem, and to Luis Ferroni for calling their attention to the applications discussed in subsection ``Carath\'eodory Rank of Matroid Base Polytopes.''

 Tamás Schwarcz was supported by the \'{U}NKP-22-3 and \'{U}NKP-23-4 New National Excellence Program of the Ministry for Culture and Innovation from the source of the National Research, Development and Innovation Fund. This research has been implemented with the support provided by the Lend\"ulet Programme of the Hungarian Academy of Sciences -- grant number LP2021-1/2021, and by Dynasnet European Research Council Synergy project (ERC-2018-SYG 810115).

\bibliographystyle{abbrv}
\bibliography{regular}

\end{document}